\documentclass[12pt]{amsart}

\usepackage{times,amsfonts,amsmath,amstext,amsbsy,amssymb,
amsopn,amsthm,upref,eucal}

\usepackage{url}

\usepackage{graphicx}

\newcommand \arccosh {{\rm arccosh}}

\newcommand \la {\lambda}
\newcommand \Var {\mathrm {Var}}
\newcommand \ee {{\mathbb E}}

\newcommand \Y {{\mathbb Y}}

\newcommand \SIN {{\mathcal S}}

\newcommand \J {{\mathcal J}}

\newcommand \Prob {{\mathbb P}}

\newcommand \pln {{\mathbb Pl}^{(n)}}

\newtheorem{theorem}{Theorem}[section]
\newtheorem{lemma}[theorem]{Lemma}

\newtheorem{corollary}[theorem]{Corollary}
\newtheorem{proposition}[theorem]{Proposition}

\title[ The Vershik-Kerov Conjecture]{On the Vershik-Kerov Conjecture
Concerning \\the~
Shannon-McMillan-Breiman~Theorem \\
for the
Plancherel Family of Measures \\
on the Space of Young Diagrams.}
\author{Alexander I. Bufetov}
\address{Rice University, Houston}
\address{The Steklov Institute of Mathematics, Moscow}
\address {The Institute for Information Transmission Problems, Moscow}
\address{The Higher School of Economics, Moscow}
\address{The Independent University of Moscow}
\date{}

\begin{document}
\maketitle
\begin{center}
{\it To Alesha}
\end{center}

\tableofcontents
\section{Introduction.}

\subsection{The Vershik-Kerov Conjecture.}

Let $n\in {\mathbb N}$ and let $\Y_n$ be the set of Young diagrams with $n$ cells. For $\la\in\Y_n$ let $\dim\la$ be the dimension of the
irreducible representation of the symmetric group on $n$ elements
corresponding to $\la$. The {\it Plancherel} probability
measure $\pln$ on $\Y_n$ is given by the formula
$$
\pln(\la)=\frac{\dim^2\la}{n!}.
$$

In 1985 Vershik and Kerov \cite{VK}  showed that there exist two positive constants
$\alpha_1, \alpha_2$ such that
$$
\lim\limits_{n\to\infty} \pln\{\la\in\Y_n: \alpha_1\sqrt{n}\leq -\log\pln(\la)\leq \alpha_2\sqrt{n}\}=1.
$$
and conjectured that the sequence of random variables $\frac{-\log\pln(\la)}{\sqrt{n}}$
converges to a constant according to the Plancherel measure.

The main result of this paper is the proof of the Vershik-Kerov conjecture:
\begin{theorem}
\label{main}
There exists a constant $H>0$ such that for any $\varepsilon>0$
we have
\begin{equation}
\label{vkconjecture}
\lim\limits_{n\to\infty} \pln\left\{\la\in\Y_n:
\left|H+\frac{\log\pln(\la)}{\sqrt{n}}\right|\leq \varepsilon\right\}=1.
\end{equation}
\end{theorem}
Theorem \ref{main} immediately implies
\begin{corollary}[Asymptotic Equidistribution for the Plancherel Measure]
For any $\varepsilon>0$ there exists $n_0>0$ such that for any $n\in {\mathbb N}$, $n>n_0$, there
exists a subset ${\mathbb Y}_n(\varepsilon)\subset {\mathbb Y}_n$ with the following properties:
\begin{enumerate}
\item The cardinality $\#{\mathbb Y}_n(\varepsilon)$ of the set ${\mathbb Y}_n(\varepsilon)$
satisfies the inequality
$$
e^{(H-\varepsilon)\sqrt{n}}\leq \#{\mathbb Y}_n(\varepsilon)\leq e^{(H+\varepsilon)\sqrt{n}}.
$$
\item For each $\la\in{\mathbb Y}_n(\varepsilon)$ we have
$$
e^{-(H+\varepsilon)\sqrt{n}}\leq \pln(\la)\leq e^{-(H-\varepsilon)\sqrt{n}}.
$$
\end{enumerate}

\end{corollary}

Vershik and Kerov have suggested to call $H$
the {\it entropy} of the Plancherel measure.
Numerical experiments allowing to estimate the entropy $H$  of the Plancherel measure are given in \cite{VP}.
An explicit formula for $H$ is given below in (\ref{formulaalpha}).
\subsection{The Shannon-McMillan-Breiman Theorem}
The term ``entropy'' is suggested by the following analogy. Let $W(N)$ be the set of all binary words of length $N$. Take $p \in (0,1)$, and let $\mathbb P_{N,p}$ be the Bernoulli measure which to a word $w \in W(N)$ with $k$ zeros assigns the weight
\[
\mathbb P_{N,p} (w) \;=\; p^k \cdot (1-p)^{N-k}.
\]
Let
\[
H(p) \;=\; -p\log p - (1-p)\log(1-p)
\]
be the entropy of the Bernoulli measure. Shannon's Theorem then says that for any $p\in (0,1)$ and any $\varepsilon>0$ we have
\[
\lim_{N \to \infty} \mathbb P_{N,p} \left( \left\{
w \in W(n): \, \left|
\frac{\log \mathbb P_{N,p}(w)}{N} + H(p) \right|>\varepsilon
\,\right\} \right) =0.
\]

Using Kolmogorov's Strong Law of Large Numbers,
Shannon's Theorem can be strengthened to a pointwise statement, the Shannon-McMillan-Breiman Theorem (for a detailed discussion of the Shannon-McMillan-Breiman Theorem see e.g. \cite{CSF}, \cite{coverthomas}, \cite{petersen}). Similarly, Vershik and Kerov have given a pointwise analogue of their conjecture. Let $\mathbb T(\mathbb Y)$ be the space of infinite {\it Young tableaux}, that is, infinite directed paths in the Young graph starting at the origin.
The sequence $\left({\pln}\right)_{n\in {\mathbb N}}$ gives rise to a natural Markov measure on $\mathbb T(\mathbb Y)$; that
measure is denoted $\mathbb Pl$ and called the Plancherel measure on $\mathbb T(\mathbb Y)$ (see \cite{VK} for details). Vershik and Kerov conjectured that for $\mathbb Pl$-almost all paths $$(\lambda^{(n)})_{n \in \mathbb N} \in \mathbb T(\mathbb Y)$$ we have
\[
\lim_{n \to \infty} \frac{-\log \mathbb Pl^{(n)} (\lambda^{(n)})}{\sqrt{n}} \;=\; H.
\]
The pointwise conjecture remains open.


 Convergence in $L_p$ for $p<\infty$, on the other hand, can be obtained simultaneously with the convergence in measure. Let $\ee_{\pln}$ stand for the expectation with respect to the Plancherel measure.
\begin{corollary}
\label{maincorol}
There exists a constant $H>0$ such that for any $p$, $0<p<\infty$,
we have
$$
\lim\limits_{n\to\infty} \ee_{\pln}
\left|H+\frac{\log\pln(\la)}{\sqrt{n}}\right|^p=0.
$$
\end{corollary}
Indeed, by the Euler-Hardy-Ramanujan Formula, the number of Young diagrams with $n$ cells
does not exceed  $\exp(2\pi \sqrt{n}/\sqrt{6})$, whence
$$
\pln\left\{\la:  \frac{-\log\pln(\la)}{\sqrt{n}}>K\right\}\leq
\exp((-K+2\pi/\sqrt{6})\sqrt{n}),
$$
and Corollary \ref{maincorol} is immediate from Theorem \ref{main} .
\subsection{Outline of the Proof of Theorem \ref{main}.}
The first step,
due to Vershik and Kerov \cite{VK}, is a variational formula for the normalized logarithm of the Plancherel measure (see Subsection \ref{vkvf}).
Using the hook formula, Vershik and Kerov represent the normalized
logarithm of the Plancherel measure as a special double integral, called {\it the hook integral}.
The hook integral admits a unique minimum -- the Vershik-Kerov-Logan-Shepp limit shape. The Vershik-Kerov variational formula  (\ref{vkformula}) is an explicit expression for the quadratic variation of the hook integral.

The next step is the Theorem established, independently and simultaneously, by Borodin, Okounkov and Olshanski \cite{BOO} and
Johansson \cite{johansson}, which claims that the poissonization of the Plancherel measure is the discrete Bessel determinantal point process. Using this Theorem, Borodin, Okounkov and Olshanski showed that local patterns in the bulk of a Plancherel Young diagram are governed by the discrete sine-process.

The Vershik-Kerov Variational Formula has two types of terms: the {\it local} terms and the {\it nonlocal} terms.
For the local terms, the Borodin-Okounkov-Olshanski Theorem is averaged along the boundary of the Young diagram, and it is shown that the normalized number of appearances of a given local pattern in a Young diagram converges to a constant with respect to the Plancherel measure (Lemma \ref{localergodic}). In particular, for $k \in \mathbb N$, it is shown that the normalized number of cells with hook length $k$ converges to a constant according to the Plancherel measure (Lemma \ref{hooklength}).
The proof relies on a simple upper estimate for the decay of correlations of the Plancherel measure (Lemma \ref{correlplan}).

The final step is to show that the nonlocal terms of the Vershik-Kerov formula converge to $0$ according to the Plancherel measure (Lemma \ref{tailh}, proved in Section \ref{prooftailh}).
The proof relies on upper estimates for the variance of the Bessel point process and the Plancherel measure, which are obtained using the classical contour integral representations for Bessel functions and the Okounkov contour integral representation for the discrete Bessel kernel (Section \ref{secbessel}).
\subsection{Acknowledgements.} Grigori Olshanski posed the problem to me
and suggested the poissonization approach; I am deeply grateful to him.
I am deeply grateful to Alexei  Borodin who suggested the use
of Okounkov's contour integral for the discrete Bessel kernel.
I am deeply grateful to Elena Rudo for her careful reading of the manuscript and
for many very helpful suggestions on improving the presentation.
I am deeply grateful to Sevak Mkrtchyan, Fedor Petrov, Alexander Soshnikov,
Konstantin Tolmachov and Anatoly M. Vershik
for helpful discussions. I am deeply grateful to the referee for many useful comments.
I am deeply grateful to Nikita Kozin for typesetting parts of the manuscript.
This work was supported in part by an Alfred P. Sloan Research Fellowship,
by the Grant MK-4893.2010.1 of the President of the Russian Federation,
by the Programme on Mathematical Control Theory of the Presidium of the Russian Academy of Sciences,
by the Programme 2.1.1/5328 of the Russian Ministry of Education and Research,
by the RFBR-CNRS grant 10-01-93115, by the RFBR grant 11-01-00654,
by the Edgar Odell Lovett Fund at Rice University and by the National Science Foundation under grant DMS~0604386.

\section{The Vershik-Kerov Variational Formula}
\subsection{The Limit Shape of Plancherel Young Diagrams.}

Take a Young diagram $\la=(\la_1, \la_2, \dots)$ (setting $\la_i=0$ for all large $i$).
Introduce a piecewise-linear function $\Phi_{\la}$
in the following way: we set $\Phi_{\la}^{\prime}|_{(k,k+1)}=-1$ if $k=\la_i-i$ for some $i$, we set
$\Phi_{\omega}^{\prime}|_{(k,k+1)}=1$ otherwise, and we require that the equality $\Phi_{\la}(t)=|t|$
hold for all sufficiently large $t$
(it is easy to see
that the continuous function $\Phi_{\la}$ is uniquely defined by these requirements;
it is differentiable except at integer points).

The function $\Phi_{\la}$ admits the following combinatorial interpretation.
Assume that the cells of our diagram are squares with diagonal $2$.
Following Vershik and Kerov \cite{VK}, rotate the diagram $\la$ by $\pi/4$;
the boundary of the rotated diagram forms the graph of
$\Phi_{\la}$, while ``beyond'' the diagram, for all sufficiently large $|t|$,  we have $\Phi_{\la}(t)=|t|$ (see Fig. 1 on p. 482 in \cite{BOO}).

Following Vershik and Kerov, introduce the function
$$
\Omega(t)=
\begin{cases}
\frac 2{\pi}(t\arcsin(t/2)+\sqrt{4-t^2}),&\text{if $|t|\leq 2$;}\\
                  |t|,&\text{ if $|t|> 2$, }
\end{cases}
$$
and denote
$$
F_{\la}(t)={\Phi}_{\la}(t)-\sqrt{n}\Omega(t/\sqrt{n}).
$$
By definition, the function $F_{\la}$ has compact support.
The functions ${\Phi}_{\la}$ and $\Omega$ are Lipschitz with
constant $1$, therefore the function $F_{\la}$  is Lipschitz
with constant $2$.

Vershik and Kerov \cite{VK77} and, independently and simultaneously,
Logan and Shepp \cite{loganshepp} have shown that for any $\varepsilon>0$
we have
$$
\lim\limits_{n\to\infty} \pln\{\la\in\Y_n:  |{F}_{\la}(t)/\sqrt{n}|
\leq \varepsilon\}=1.
$$

\subsection{The Quadratic Variation of the Hook Integral}
\label{vkvf}

Recall that the {\it hook length } of a cell in a Young diagram is
the number of cells to the right of it  and under it (including the cell itself)
and let $h_k(\la)$ stand for the number of cells in $\la$ with hook length $k$.

Denote
$$
||F_{\la}||_{1/2}=\int\limits_{-\infty}^{\infty}
\int\limits_{-\infty}^{\infty} \left(\frac{F_{\la}(t)-F_{\la}(s)}{t-s} \right)^2dtds=
2\int\limits_{0}^{\infty}
\int\limits_{-\infty}^{\infty} \left(\frac{F_{\la}(t+h)-F_{\la}(t)}{h} \right)^2dtdh.
$$
The Vershik-Kerov  Variational Formula(see \cite{VK}, Lemma 1 and formulas (5), (8), (9))
is the  equality
\begin{multline}
\label{vkformula}
\frac{-\log \pln(\la)}{\sqrt{n}}=\frac 1{\sqrt{n}}
\sum\limits_{k=1}^{\infty} h_k(\la)
\left(\sum\limits_{l=1}^{\infty} \frac1{l(l+1)(2l+1)k^{2l}}\right)+
\frac 1{8\sqrt{n}}||F_{\la}||_{1/2}+\\
+\frac1{\sqrt{n}}\int\limits_{|t|\geq 2\sqrt{n}} F_{\la}(t)\arccosh\left(\frac{t}{2\sqrt{n}}\right)dt -\varepsilon_n,
\end{multline}
where $\varepsilon_n$ only depends on $n$ (not on $\la$)
and tends to $0$ as $n\to\infty$.

It will be convenient for us to adopt the following terminology.
Assume that for each $n\in {\mathbb N}$ we are given a random variable
$\xi_n$ on $\Y_n$. If there exists $\beta$ such that for any
$\varepsilon>0$ we have
$$
\lim\limits_{n\to\infty} \pln\{\la\in\Y_n: |\xi_n(\la)-\beta|
\leq \varepsilon\}=1,
$$
then we say that $\xi_n$ converges to the constant  $\beta$ according to the
Plancherel measure.

If
$$
\lim\limits_{n\to\infty} \pln\{\la\in\Y_n: \xi_n(\la)<\beta\}=1,
$$
we say that $\xi_n$ is asymptotically majorated by $\beta$
according to the Plancherel measure.

We shall analyze the terms of the Vershik-Kerov Variational Formula one by one.

\begin{lemma}
\label{hooklength}
For any $k\in {\mathbb N}$ the random variables $\frac{h_k(\la)}{\sqrt{n}}$
converge to the  constant $\frac{32k^2}{(4k^2-1)\pi^2}$ according to the
Plancherel measure.
\end{lemma}

\begin{lemma}
\label{lochone}
For any $h_0>0$ the random variables
$$
\frac 1 {\sqrt{n}}\int\limits_0^{h_0}\int\limits_{-\infty}^{\infty}
\left(\frac{F_{\la}(t+h)-F_{\la}(t)}{h} \right)^2 dt dh
$$
converge to a constant
according to the
Plancherel measure.
\end{lemma}

\begin{lemma}
\label{tailh}
For any $\varepsilon>0$ there exists $h_0>0$ such that the random
variables
$$
\frac 1 {\sqrt{n}}\int\limits_{h_0}^{\infty}\int\limits_{-\infty}^{\infty}
\left(\frac{F_{\la}(t+h)-F_{\la}(t)}{h} \right)^2 dt dh
$$
are asymptotically majorated by $\varepsilon$
according to the Plancherel measure.
\end{lemma}

\begin{lemma}
\label{lishnijterm}
The random variables
\begin{equation}
\label{vybrosterm}
\frac 1{{\sqrt{n}}}\int\limits_{|t|\geq 2\sqrt{n}} F_{\la}(t)\arccosh\left(\frac{t}{2\sqrt{n}}\right)dt
\end{equation}
converge to $0$ according to the Plancherel measure.
\end{lemma}

As before,  let $\la_1$ be the length of the first row of $\la$, and
let $\la^{\prime}_1$ be the length of the first column of $\la$.
Vershik and Kerov \cite{VK} established that
for any $\varepsilon>0$ we have
\begin{equation}
\label{trivboundlen}
\lim\limits_{n\to\infty}\pln\left(  \{\la: \la_1<(2+\varepsilon)\sqrt{n},
\la^{\prime}_1<(2+\varepsilon)\sqrt{n}\}                     \right)=1.
\end{equation}

Theorem \ref{main} is now immediate from the bound (\ref{trivboundlen}), the Vershik-Kerov Variational Formula
and the Lemmas \ref{hooklength}, \ref{lochone}, \ref{tailh}, \ref{lishnijterm}.

We proceed to the proof of the Lemmas.

\subsection{Proof of Lemma \ref{lishnijterm}.}

We shall need a more precise estimate than (\ref{trivboundlen}).
\begin{proposition}
\label{ydelta}
\begin{enumerate}
\item For any $\delta_0>1/6$ there exists  constants $C>0$, ${\tilde \gamma}>0$ such that for all
$\delta$ satisfying $\delta_0\leq \delta\leq 1/2$
we have
\begin{equation}
\lim\limits_{n\to\infty}\pln\left(\{\la: \la_1>2\sqrt{n}+n^{\delta} \ {\rm or}  \
\la^{\prime}_1>2\sqrt{n}+n^{\delta}\}\right)\leq C\exp\left({-{\tilde \gamma}n^{3\delta/2-1/4}}\right).
\end{equation}

\item For any $\varepsilon_1>0$ there exists ${\tilde \gamma}_1>0$ depending only on $\varepsilon_1$ such that for
any $\varepsilon>\varepsilon_1$ we have
 \begin{equation}
\lim\limits_{n\to\infty}\pln\left(\{\la: \la_1>(2+\varepsilon)\sqrt{n} \ {\rm or}  \
\la^{\prime}_1>(2+\varepsilon)\sqrt{n}\}\right)\leq C\exp\left({-{\tilde \gamma}_1\varepsilon n}\right).
\end{equation}

\end{enumerate}
\end{proposition}

Proposition \ref{ydelta} is well-known.
For completeness of the exposition a proof is given below (see Proposition \ref{expest}).

Now, using Proposition \ref{ydelta}, choose $\delta> \frac{1}{6}$ and assume that
\[
F_\lambda (t) = 0 \quad \text{for} \quad \left| t \right| > 2 \sqrt{n} + n^\delta
\]
In this case for $\left| t \right| \in [2\sqrt{n},\, 2\sqrt{n} + n^\delta]$ we have:
\[
\left| F_\lambda (t) \right| \le 2n^\delta, \qquad \left| \arccosh \left( \frac{t}{2\sqrt{n}} \right) \right| \;\le\; 2 n^{\frac{\delta}{2} - \frac{1}{4}},
\]
whence
\[
\frac{1}{\sqrt{n}} \left| \;
\int\limits_{\left| t\right| \ge 2 \sqrt{n}}
F_\lambda (t) \arccosh \left( \frac{t}{2\sqrt{n}} \right)\, dt \, \right|
\;\le\; 32 n^{\frac{5\delta}{2}-\frac{3}{4}},
\]
and, as soon as $\delta < \frac{3}{10}$, we are done.

We proceed to the analysis of the remaining terms.

\section{Poissonization.}

\subsection{Diagrams and Sequences.}
Let $\Omega_2=\{0,1\}^{{\mathbb
Z}}$ be the space
of bi-infinite sequences of the symbols $0$ , $1$:
$$
\Omega_2=\{\omega=\dots \omega(-n)\dots \omega(n) \dots,
\ \omega(n)\in \{0,1\}\}.
$$

To a sequence $\omega$ we assign a continuous piecewise-linear function $\Phi_{\omega}$
in the following way: we set $\Phi_{\omega}(0)=0$, $\Phi_{\omega}^{\prime}|_{(k,k+1)}=1$
if $\omega_k=0$, $\Phi_{\omega}^{\prime}|_{(k,k+1)}=-1$ if $\omega_k=1$
(it is easy to see
that the continuous function $\Phi_{\omega}$ is uniquely defined by these requirements;
it is differentiable except at integer points).

Take a Young diagram $\la=(\la_1, \la_2, \dots)$ (setting $\la_i=0$ for all large $i$) and
introduce a sequence
$c(\la)\in\Omega_2$  by the rule
$c_k(\la)=1$ if  $k=\la_i-i$
for  some $i$ and  $c_k(\la)=0$ otherwise. It is clear from the definitions that
the difference $\Phi_{\la}-\Phi_{c({\la})}$ is a constant.

Take  an integer vector $\vec{m}=(m_1, \dots, m_r)$ all whose coordinates are distinct and for a
Young diagram $\la$ denote
$$
c_{\vec{m}}(\la)=c_{m_1}(\la)\dots c_{m_r}(\la).
$$
Similarly, for $\omega\in \Omega_2$ write
$$
c_{\vec{m}}(\omega)=\omega_{m_1}\dots \omega_{m_r}.
$$

In what follows, when we speak of integer vectors, we shall always assume
that all their coordinates are distinct.

\subsection{The Bessel Point Process.}
Set
$$
\Y=\bigcup\limits_{n=1}^{\infty} \Y_n,
$$
and for $\eta>0$ let
$$
Pois_{\eta}=\exp(-\eta)\sum\limits_{n=0}^{\infty} \frac{\eta^n}{n!} \pln
$$
be the $\eta$-poissonized Plancherel measure on $\Y$.

At the centre of our argument lies a theorem obtained by Borodin, Okounkov and
Olshanski in \cite{BOO} and Johansson in \cite{johansson} which states that
the measure $Pois_{\eta}$ naturally induces the Bessel
determinantal point process on the space of sequences of
two symbols. We proceed to the exact formulation and start by recalling the definition of a determinantal
point process on $\Omega_2$ (for a more detailed
exposition, see, e.g., \cite{soshnikov}).

Let ${\mathcal K}: {\ell}_2({\mathbb Z})\to {\ell}_2({\mathbb Z})$ be
a self-adjoint positive contraction, or, in other words, a self-adjoint linear operator satisfying
$$
0\leq \langle f, {\mathcal K}f\rangle\leq \langle f,f\rangle.
$$
Set
$$
{\mathcal K}(x,y)=\langle {\mathcal K}\delta_y, \delta_x\rangle.
$$
There exists a unique probability measure $\Prob_{{\mathcal K}}$ on $\Omega_2$ such that
$$
\ee_{\Prob_{{\mathcal K}}}(c_{{\vec m}})=\det \left({\mathcal K}(m_i,
m_j)\right)\big|_{i,j=1, \dots, r}.
$$

Now set $\eta=\theta^2$ (assuming $\theta>0$) and, for $x\neq y$, write
$$
{\mathcal J}({\theta^2}; x,y)=
{\theta}\frac{J_x(2\theta)J_{y+1}(2\theta)-J_{x+1}(2\theta)J_{y}(2\theta)}{x-y}.
$$
The expression ${\mathcal J}({\theta^2}; x,x)$ is defined using the l'Hospital Rule.
The kernel ${\mathcal J}({\theta^2})$ is called {\it the discrete
Bessel kernel}, and the
resulting measure
$\Prob_{{\mathcal J}({\theta^2})}$  on $\Omega_2$ is called {\it the Bessel point process}.

Recall that to a Young diagram $\la$ we have assigned a sequence  $c(\la)\in\Omega_2$.
Slightly abusing notation, we denote the push-forward of the
measure $Pois_{{\eta}}$ on $\Y$ under the map $\la\to c(\la)$ by the same symbol $Pois_{\eta}$.

The theorem of Borodin, Okounkov and Olshanski \cite{BOO} and Johansson \cite{johansson}
states that the measure $Pois_{\eta}$
defined above is precisely the Bessel point process with parameter
$\eta=\theta^2$.

\subsection{Depoissonization.}

Informaton about the Plancherel measure will be derived from the corresponding properties of the Bessel point process
with the use of the following lemma of Borodin, Okounkov and Olshanski
(a slight modification of Lemma 3.1 in \cite{BOO}).

\begin{lemma}[ Borodin, Okounkov, Olshanski]
\label{depoisson}
Let $0<\alpha<1/4$.
Let $\{f_n\}$ be a sequence of entire functions
\begin{equation}
\label{fnz}
f_n(z)=\exp(-z)\sum_{k\geq 0} \frac{f_{nk}}{k!}z^k, \ n=1, \dots
\end{equation}
and assume that there exist constants $f_{\infty}$, $\gamma$, $C_1$, $C_2$ such that
\begin{enumerate}
\item $\max\limits_{|z|=n} |f_n(z)|\leq C_1\exp(\gamma\sqrt{n})$;
\item $\max\limits_{|z-n|<n^{1-\alpha}}|f_n(z)-f_{\infty}|\exp(-\gamma|z-n|/\sqrt{n})\leq C_2.$
\end{enumerate}
Then there exists a constant $C=C(\gamma, C_1, C_2)$ such that
for all $n>0$ we have
$$
|f_{nn}-f_{\infty}|\leq C.
$$
\end{lemma}

The proof is identical to the proof of Lemma 3.1 in \cite{BOO} except that
$o(1)$ in the last two formulas on page 495 must be replaced by $O(1)$.

To use Lemma \ref{depoisson} we must allow complex values of the poissonization
parameter $\theta^2$: in this case, expressions
such as  $\ee_{{\J}(\theta^2)}$ are understood formally (by analytic
continuation).

Lemma \ref{depoisson} can be equivalently reformulated as follows.
\begin{lemma} \label{depoismod}
Let $\delta>0$ be arbitrary, let $\alpha$ satisfy $0<\alpha< \frac{1}{4}$. Assume that there exist constants
$
f_\infty, \; \gamma_1,\, \gamma_2,\, \gamma_3, \, C_1,\, C_2,\, C_3 \; >0
$
such that
\begin{align*}
\text{1)} \qquad &\max_{\left| z - n \right| < n^\delta} \left| f_n(z) - f_\infty \right| e^{-\frac{ \gamma_1 \left| z - n \right|}{\sqrt{n}} } \;\le\; C_1,\\
\text{2)} \qquad &\max_{\left| z - n \right| < n^{1-\alpha}} \left| f_n(z) \right| e^{-\frac{ \gamma_2 \left| z - n \right|}{\sqrt{n}} } \;\le\; C_2, \\
\text{3)} \qquad &\max_{\left| z \right| = n} \left| f_n(z) \right| \;\le\; C_3 e^{\gamma_3 \sqrt{n}}.
\end{align*}
Then there exists a constant $C = C(\gamma_1, \, \gamma_2, \, \gamma_3,\, C_1, \, C_2, \, C_3)$ such that for all $n>0$ we have
\[
\left| f_{nn} - f_\infty \right| < C.
\]
\end{lemma}

\begin{corollary} \label{depoisan}
Assume that the sequence $f_n$ of entire functions defined by (\ref{fnz}) satisfies conditions 2) and 3) of Lemma \ref{depoismod}. Let $\widetilde C_1 >0$, and let $a_n$ be a sequence of positive numbers satisfying $\left| a_n \right| \le \widetilde C_1$. If
\[
\max_{\left| z - n \right| < n^\delta} \left| f_n(z) - f_\infty \right| e^{-\frac{ \gamma_1 \left| z - n \right|}{\sqrt{n}} } \;\le\; C_1 a_n,
\]
then for all $n>0$ we have:
\[
\left| f_{nn} - f_\infty \right| \;\le\; C a_n.
\]
Here $C$, again, is a constant depending only on $\gamma_1,\, \gamma_2,\, \gamma_3, \, C_1,\,\widetilde C_1,\, C_2,\, C_3$.
\end{corollary}

\begin{proof}
Follows by applying Lemma \ref{depoismod} to the sequence
\[
\frac{f_n(z)-f_\infty}{a_n}.
\]
\end{proof}

We shall be mainly concerned with depoissonization of various polynomials of Bessel functions, and it is useful to note that in this case Conditions 2 and 3 in Lemma \ref{depoismod} and Corollary \ref{depoisan} hold automatically. More precisely, we have the following proposition.

\begin{proposition}
Let $K>0$, $k\in {\mathbb N}$. Let $P$ be a polynomial in $k$ variables. Let
\[
x_n^{(1)}, \, \ldots, \, x_n^{(k)}, \ n\in {\mathbb N},
\]
be $k$ sequences of integers satisfying
\[
\left| x_n^{(1)} \right|, \, \ldots, \, \left| x_n^{(k)} \right| \le K\sqrt{n}.
\]
Then there exist constants
\[
C_1, \, C_2,\, \gamma_1, \, \gamma_2 \; > \; 0
\]
depending only on $k,\, K$ and $P$ such that
\begin{align*}
&\text{1)} \quad
\max_{\left| z \right| = n}
\left|
P \left(\, J_{x_n^{(1)}} (2\sqrt{z}\,), \, \ldots, \, J_{x_n^{(k)}} (2\sqrt{z}\,)
\right) \right| \;\le\; C_1 e^{\gamma_1 \sqrt{n}}.
\\
&\text{2)} \quad
\max_{\left| z-n \right|< n^{1-\alpha}}
\left|
P \left(\, J_{x_n^{(1)}} (2\sqrt{z}\,), \, \ldots, \, J_{x_n^{(k)}} (2\sqrt{z}\,)
\right) \right| e^{-\frac{\gamma_2 \left| z - n\right|}{\sqrt{n}}} \;\le\; C_2.
\end{align*}
\end{proposition}

\noindent
The Proposition is immediate from the contour integral representation of the Bessel functions.

We also note that in the depoissonization arguments that  follow, weaker assumptions on $\theta$ than those of \cite{BOO}
are sufficient: namely, we shall always assume that
$\theta=\sqrt{z}$ satisfies
$$
\left|\frac{\theta}{\sqrt{n}}-1\right|\leq \varepsilon_0,
$$
where $\varepsilon_0$ is sufficiently small.

\subsection{The Debye Asymptotics.}


For depoissonization we need the asymptotics of
Bessel functions when both order and argument are large. First results
of this type are due to Carlini;
we shall use the asymptotics due to Debye, following the exposition
by Watson \cite{watson}.

Take $\varepsilon>0$. Set
\begin{equation}
N_n^{(\varepsilon)}=\{k\in {\mathbb Z}: \frac{|k|}{\sqrt{n}}<2-\varepsilon\}.
\end{equation}

Take $x\in  N_n^{(\varepsilon)}$.  Then there exists $\varepsilon_0>0$ depending only on $\varepsilon$
such that for any $\theta\in {\mathbb C}$ satisfying $|\theta/\sqrt{n}-1|<\varepsilon_0$ the following is
true. Introduce $u$ by the formula $\cos u=x/2\theta$, $0<\Re(u)<\pi$.
Then we have the following representation for Bessel functions,
asymptotic in the sense of Poincar{\'e}:
\begin{equation}
\label{debyewatson}
J_x(2\theta)=\frac{\cos\big(2\theta(\tan u -u)-\frac{\pi}{4}\big)}{\sqrt{\pi\theta\tan u}}
\big(1+\sum_{m=1}^{\infty} \frac{\alpha_m(u)}{\theta^m} \big),
\end{equation}
where for any $\varepsilon^{\prime}>0$ there exists $\delta^{\prime}>0$ such
that all $\alpha_m(u)$ are holomorphic in $u$ in the
strip $[\varepsilon^{\prime}, \pi-\varepsilon^{\prime}]\times [-\delta^{\prime}, \delta^{\prime}]$.

\section{Local Patterns in Plancherel Young Diagrams.}
\subsection{The Discrete Sine-Process.}
Take $a\in (-2,2)$ and
introduce {\it the discrete  sine-kernel} by the formula
$$
{\mathcal S}(k,a)=\begin{cases}
\frac{\sin(\arccos(a/2)k)}{\pi k}, &\text{if $k\neq 0$;}\\
                 \frac{\arccos(a/2)}{\pi},&\text{if $k=0$.}
\end{cases}
$$

Introduce a measure $\SIN(a)$ on $\Omega_2$ by setting
\begin{equation}
\ee_{\SIN(a)}(c_{\vec{m}})=\det \left({\mathcal S}(m_i-
m_j,a)\right)\big|_{i,j=1, \dots, r}.
\end{equation}

The measure $\SIN(a)$ is called {\it the discrete sine-process}.

For $x\in {\mathbb Z}$, ${\vec m}\in {\mathbb Z}^r$, denote $x+{\vec m}=(x+m_1, \dots, x+m_r)$.

The theorem of Borodin, Okounkov and Olshanski \cite{BOO}
says that for any $a\in (-2,2)$, any integer vector ${\vec m}$ and any sequence
$x_n\in {\mathbb Z}$ satisfying
$$
\lim\limits_{n\to\infty} \frac{x_n}{\sqrt{n}}=a,
$$
we have
\begin{equation}
\lim\limits_{n\to\infty} \ee_{\pln}\left( c_{x_n+\vec{m}}(\la)\right)=\ee_{\SIN(a)}\left(c_{{\vec m}}\right).
\end{equation}

\subsection{The Variance of the Discrete Sine-Process.}

We shall need the following simple estimate.

\begin{proposition}
\label{varsinproc}
There exists a positive constant $C$ such that for any $s\in[0,1]$, any $h>1$ and any $a\in(-2,2)$ we have
$$\mathbb{E}_{{\mathcal S}(a)}\left(\Phi_{\omega}(s+h)-\Phi_{\omega}(s)-\frac{2}{\pi} \arcsin\left(\frac{a}{2}\right)h\right)^2\leq C\left(1+\log h\right)\,.$$
\end{proposition}

\begin{proof}
First, recall that for any $a\in(-2,2)$ the operator $\mathcal{S}_a :l_2(\mathbb{Z})\rightarrow l_2(\mathbb{Z})$ given by the formula
$$\mathcal{S}_a f(x)=\sum\limits_{k\in\mathbb{Z}} \mathcal{S}(k,a) f(x+k)\,,$$
is an orthogonal projection (this is easy to check by taking the Fourier transform).

Second, recall the well-known
\begin{proposition}
Let $\mathcal{K}:l_2(\mathbb{Z})\rightarrow l_2(\mathbb{Z})$
$$\mathcal{K} f(x)=\sum\limits_{y\in {\mathbb Z}}\mathcal{K}(x,y)f(y)\,,$$
be an orthogonal projection, and let $\mathbb{P}_{\mathcal{K}}$ be the corresponding determinantal measure on $\Omega_2$. Then for any $k_1,k_2\in\mathbb{Z},\;k_1<k_2,\,$ we have
$${\rm Var}_{{\mathbb P}_{\mathcal{K}}} \left(\sum\limits_{n=k_1}^{k_2}c_n\right)=\sum\limits_{x\in[k_1,k_2]} \sum\limits_{y\notin[k_1,k_2]}\left|\mathcal{K}(x,y)\right|^2$$
\end{proposition}
The proof is a straightforward computation using the formula
$$\mathcal{K} (x,x)=\sum\limits_{y\in\mathbb{Z}}\left|\mathcal{K}(x,y)\right|^2\,,$$
which holds for any $x\in\mathbb{Z}$.

In the remainder of the proof, $C$ stands for a positive constant that does not depend on $a\in(-2,2)$.

 By definition, for any $k\neq0$ and all $a\in(-2,2)$ we have $\left|\mathcal{S}(k,a)\right|\leq\frac{1}{|k|}$. Therefore, for any $N>0$, we have
$${\rm Var}_{\mathcal{S}(a)} \left(\sum\limits_{n=0}^{N}c_n\right)\leq C(1+\log N).$$
Since $\Phi_{\omega}(0)=0$ and $$\Phi_{\omega}(n+1)-\Phi_{\omega}(n)=1-2c_n(\omega)\,,$$
for any $N>0$ we have $$\Phi_{\omega}(N)=N-2\sum\limits_{n=0}^{N-1}c_n(\omega)\,.$$

It follows that
$$\mathbb{E}_{{\mathcal S}(a)}\Phi_{\omega}(N)=\frac{2}{\pi}N\arcsin \left(\frac{a}{2}\right)\,,$$
$${\rm Var}_{\mathcal{S}(a)}\Phi_{\omega}(N)\leq C(1+\log N).$$
In other words,
$$\mathbb{E}_{{\mathcal S}(a)}\left(\Phi_{\omega}(N)-\frac{2}{\pi}N\arcsin \left(\frac{a}{2}\right)\right)^2\leq C(1+\log N).$$

 Recall that the function $\Phi_{\omega}$ is Lipschitz with constant 1.  Using the inequality $(c+d)^2\leq 2(c^2+d^2)$, for any $s\in(0,1)$ and any $h>1$ we finally obtain
$$\mathbb{E}_{\SIN(a)}\left(\Phi_{\omega}(s+h)-\Phi_{\omega}(s)-\frac{2}{\pi} \arcsin\left(\frac{a}{2}\right)h\right)^2\leq C\left(1+\log h\right),$$
which is what we had to prove.

\end{proof}

\subsection{Decay of Correlations for the Plancherel Measure.}

Borodin, Okounkov and Olshanski \cite{BOO} have also shown that if $a\neq b$ and
$$
\lim\limits_{n\to\infty} \frac{x_n}{\sqrt{n}}=a,\
\lim\limits_{n\to\infty} \frac{y_n}{\sqrt{n}}=b,
$$
then for any integer vectors $\vec{l}, \vec{m}$ we have
\begin{equation}
\lim\limits_{n\to\infty} \ee_{\pln}
\left(c_{x_n+\vec{m}}(\la)\cdot c_{y_n+\vec{l}}(\la)\right)=
\ee_{\SIN(a)}c_{\vec{m}}\cdot \ee_{\SIN(b)}c_{\vec{l}}.
\end{equation}
Distant local patterns in a Young diagram are thus asymptotically independent.
We shall need an estimate for the decay of correlations of the Plancherel
measure.


For an integer vector $\vec {m}$, let $|\vec{m}|$ stand for the maximum of absolute
values of its coordinates.
\begin{lemma}
\label{correlplan}
For any $\varepsilon>0$, $L>0$, there exists
a constant $C=C(\varepsilon, L)$ such that for any $n>0$, any
$x,y\in {\mathbb Z}$ such that
$$
\frac{|x|}{\sqrt{n}}, \frac{|y|}{\sqrt{n}}<2-\varepsilon
$$
and  any integer vectors ${\vec l}$, ${\vec m}$ satisfying
$|{\vec l}|\leq L$, $|{\vec m}|\leq L$,
we have
\begin{multline}
\label{xyplancherel}
\left|
\ee_{\pln}(c_{x+\vec{l}}\cdot c_{y+\vec{m}})-\ee_{\pln}(c_{x+\vec{l}})\cdot
\ee_{\pln}(c_{y+\vec{m}})
\right|
\leq \\
C(\varepsilon, L)\left(\left(\frac{1}{|x-y|+1}\right)^2+\frac{1}{\sqrt{n}}\right).
\end{multline}
\end{lemma}
\subsection{Frequency of Local Patterns.}
Lemma \ref{correlplan} will be used to in the next section to prove the following
\begin{lemma}
\label{localergodic}
For any continuous bounded function
$f: {\mathbb R}\to {\mathbb C}$ and any integer vector ${\vec m}$,
the sequence of random variables
\begin{equation}
\frac1{\sqrt{n}} \sum\limits_{k=-\infty}^{\infty} f\left(\frac{k}{\sqrt{n}}\right)c_{k+{\vec m}}(\la)
\end{equation}
 converges, as
$n\to\infty$,
to the constant
$$
\int\limits_{-2}^2 f(a) \ee_{{\mathcal S}(a)}c_{{\vec m}} da
$$
according to the Plancherel measure.
\end{lemma}

We shall see in the next section that Lemma \ref{hooklength} is a simple
corollary of Lemma \ref{localergodic}.
Lemma \ref{lochone} admits the following more precise version, which will also be derived from
Lemma \ref{localergodic}.
\begin{lemma}
\label{loch}
For any $h_0>0$ the random variables
$$
\frac 1 {\sqrt{n}}\int\limits_0^{h_0}\int\limits_{-\infty}^{\infty}
\left(\frac{F_{\la}(t+h)-F_{\la}(t)}{h} \right)^2 dt dh
$$
converge to the  constant
$$
\int\limits_{-2}^2 \int\limits_0^1 \int\limits_0^{h_0}\ee_{{\mathcal S}(a)} \left(\frac{\Phi_{\omega}(s+h)-\Phi_{\omega}(s)}{h}-
\frac 2{\pi}\arcsin(a/2)\right)^2 dhdsda
$$
according to the
Plancherel measure.
\end{lemma}

For the constant $H$, the entropy of the Plancherel measure, we now obtain
\begin{multline}
\label{formulaalpha}
H=\frac14 \int\limits_{-2}^2 \int\limits_0^1 \int\limits_0^{\infty} \ee_{{\mathcal S}(a)}\left(\frac{\Phi_{\omega}(s+h)-\Phi_{\omega}(s)}{h}-
\frac 2{\pi}\arcsin(a/2)\right)^2 dhdsda+\\
+\frac{32}{\pi^2}\sum\limits_{k=1}^{\infty}\sum\limits_{l=1}^{\infty}
\frac{1}{l(l+1)(2l+1)k^{2l-2}(4k^2-1)}.
\end{multline}

Convergence of the integral in $h$ is clear from Proposition \ref{varsinproc}.

\section{Proof of Lemmas \ref{correlplan}, \ref{localergodic}, \ref{loch},
 \ref{hooklength}, \ref{lochone}.}

\subsection{Decay of Correlations for the Bessel Point Process.}

Given a measure $\Prob$ on $\Omega_2$,  any natural $x, y$, and any integer vectors $\vec{l}= (l_1, \dots, l_r)$,
$\vec{m}=( m_1, \dots, m_s)$,
denote
$$
\mathrm {Cov}_{\Prob}(x, \vec{l}; y, \vec{m})=
\ee_{\Prob}(c_{x+\vec{l}}\cdot c_{y+\vec{m}})-\ee_{\Prob}(c_{x+\vec{l}})\cdot
\ee_{\Prob}(c_{y+\vec{m}}).
$$
For complex $\theta$ the expression
$\mathrm {Cov}_{\J(\theta^2)}(x, \vec{l}; y, \vec{m})$ is defined formally,
by analytic continuation. Our next aim is to estimate this quantity from above.

The representation (\ref{debyewatson}) implies, in particular, the existence of constants
$C, \gamma, \varepsilon_0$ depending only on $\varepsilon$ such that
\begin{equation}
\label{jxy}
|\theta(J_x(2\theta)J_{y+1}(2\theta)-J_y(2\theta)J_{x+1}(2\theta))|\leq
C\exp(\gamma|\theta-\sqrt{n}|),
\end{equation}
provided
$x, y\in N_n^{(\varepsilon)}$, $|\theta/\sqrt{n}-1|<\varepsilon_0$, whence, if $x\neq y$,
we have
$$
\left| \J(x,y; \theta^2)\right|\leq \frac{
C\exp(\gamma|\theta-\sqrt{n}|)}{|x-y|}. $$

The function $\J(x,y,\theta^2)$ is entire in $x,y$ and, in the same way as in $(3.7)$ on p.498 in \cite{BOO},
write
$$
\J(x,x,\theta^2)=\frac 1{2\pi} \int\limits_0^{2\pi} \J(x, x+r\exp(it), \theta^2)dt,
$$
where $r$ is arbitrary. This representation shows that
in the case $x=y$ we also have
\begin{equation}
\label{jxx}
\left| \J(x,x; \theta^2)\right|\leq C\exp(\gamma|\theta-\sqrt{n}|)
\end{equation}

We have  established  the following
\begin{lemma}
\label{correlbessel}
For any $\varepsilon>0$, $L>0$,
there exist positive constants $C=C(\varepsilon, L)$,
$\gamma=\gamma(\varepsilon, L)$
and $\varepsilon_0$ depending only on $\varepsilon$, such that for any $n>0$,
any $\theta$ satisfying $|\theta|=\sqrt{n}$, $|\theta/\sqrt{n}-1|<\varepsilon_0$, any
$x,y\in N_n^{(\varepsilon)}$, and any integer vectors ${\vec l}$, ${\vec m}$  with absolute values not exceeding $L$,
we have
\begin{equation}
\label{xybessel}
\mathrm {Cov}_{\J(\theta^2)}(x, \vec{l}; y, \vec{m})
\leq \frac{C(\varepsilon, L)\exp(\gamma|\theta-\sqrt{n}|)}{\left(|x-y|+1\right)^2}.
\end{equation}
\end{lemma}



\subsection{Proof of Lemma \ref{correlplan}.}
The Debye asymptotics (\ref{debyewatson}) immediately yields (see \cite{BOO} for details) that
for any $\varepsilon>0$ there exists $\varepsilon_0>0$ and for any $l\in {\mathbb N}$  constants
$C=C(l, \varepsilon)$, $\gamma=\gamma(l, \varepsilon)$
such that
if $\theta$ satisfies  $|\theta|=\sqrt{n}$, $|\theta/\sqrt{n}-1|<\varepsilon_0$, then for any
$x\in N_n^{(\varepsilon)}$ we have
\begin{equation}
\label{jarccos}
|\J(x,x+l, \theta^2)- {\mathcal S}(\frac{x}{2\theta}, l)|\leq
\frac{C\exp(\gamma|\theta-\sqrt{n}|)}{\sqrt{n}}.
\end{equation}

By definition of a determinantal process, the estimate (\ref{jarccos}) implies the following
\begin{proposition}
\label{ejtheta}
For any $\varepsilon>0$, $L>0$,
there exist positive constants $C=C(\varepsilon, L)$,  $\gamma=\gamma(\varepsilon, L)$
and $\varepsilon_0$ depending only on $\varepsilon$, such that for any $n>0$,
any $\theta$ satisfying $|\theta|=\sqrt{n}$, $|\theta/\sqrt{n}-1|<\varepsilon_0$, any
$x\in N_n^{(\varepsilon)}$, and any integer vector ${\vec l}$ satisfying $|{\vec l}|\leq L$,
we have
\begin{equation}
\label{jthetasin}
\left|\ee_{\J(\theta^2)}(c_{x+\vec{l}})-\ee_{{\mathcal S}(x/2\sqrt{n})}(c_{\vec{l}})\right|\leq
\frac{C\exp(\gamma|\theta-\sqrt{n}|)}{\sqrt{n}}.
\end{equation}
\end{proposition}

Depoissonizing  by Lemma \ref{depoisson}, we obtain
\begin{equation}
\label{eplnsin}
\left|\ee_{\pln}(c_{x+\vec{l}})-\ee_{{\mathcal S}(x/\sqrt{n})}(c_{\vec{l}})\right|\leq
\frac{C(\varepsilon, L)}{\sqrt{n}}.
\end{equation}

As a simple example, taking $l=0$ in (\ref{jarccos}) and depoissonizing by Lemma \ref{depoisson} yields
\begin{equation}
\label{plnarccos}
|\ee_{\pln}(c_x) - \frac 1{\pi} \arccos(x/2\sqrt{n})|\leq \frac{C(\varepsilon)}{\sqrt{n}}.
\end{equation}

Substituting (\ref{jthetasin}) into (\ref{xybessel}), we obtain
$$
\big|\ee_{\J(\theta^2)}(c_{x+\vec{l}} \cdot
c_{y+\vec{m}})-\ee_{\SIN(\frac{x}{\sqrt{n}})}(c_{x+\vec{l}})\cdot
\ee_{\SIN(\frac{y}{\sqrt{n}})}(c_{y+\vec{m}})\big|
\leq 
$$
$$
\leq C(\varepsilon, L)\exp(\gamma|\theta-\sqrt{n}|)\left(\left(\frac{1}{|x-y|+1}\right)^2+\frac{1}{\sqrt{n}}\right),
$$
whence, by the depoissonization Lemma \ref{depoisson}, we have
\begin{multline}
\left|\ee_{\pln}(c_{x+\vec{l}}\cdot
c_{y+\vec{m}})-\ee_{\SIN(\frac{x}{\sqrt{n}})}(c_{\vec{l}})\cdot
\ee_{\SIN(\frac{y}{\sqrt{n}})}(c_{\vec{m}})\right|\leq \\
\leq C(\varepsilon, L)\left(\left(\frac{1}{|x-y|+1}\right)^2+\frac{1}{\sqrt{n}}\right).
\end{multline}
 Finally, using (\ref{eplnsin}), we write
\begin{multline}
\left|\ee_{\pln}(c_{x+\vec{l}}\cdot
c_{y+\vec{m}})-\ee_{\pln}(c_{x+\vec{l}})\cdot
\ee_{\pln}(c_{y+\vec{m}})\right|
\leq  \\
\leq C(\varepsilon, L)\left(\left(\frac{1}{|x-y|+1}\right)^2+\frac{1}{\sqrt{n}}\right),
\end{multline}
and Lemma \ref{correlplan} is proved.

\subsection{Proof of Lemma \ref{localergodic}.}

As before, let $f:{\mathbb R}\to {\mathbb C}$ be continuous and bounded, let ${\vec m}=(m_1, \dots, m_r)$
be an integer vector and denote
\begin{equation}
S_n({\vec m},f,\la)=\frac1{\sqrt{n}} \sum_{k=-\infty}^{\infty} f\left(\frac{k}{\sqrt{n}}\right)c_{k+{\vec m}}(\la).
\end{equation}

Take $\varepsilon>0$ and consider a modification of the function $S_n({\vec m},f,\la)$ defined by the formula
\begin{equation}
S_n({\vec m},f,\la, \varepsilon)=\frac1{\sqrt{n}}
\sum\limits_{k: |k|\leq(2-\varepsilon)\sqrt{n}}
f\left(\frac k{\sqrt{n}}\right)c_{k+{\vec m}}(\la).
\end{equation}

\begin{proposition}
\label{epslocal}
For any $\varepsilon>0$, any continuous $f$ on $[-2,2]$
and any integer vector $\vec{m}$ we have
\begin{equation}
\label{vareps}
\lim\limits_{n\to\infty}{\ee}_{\pln} \left| S_n({\vec m}, f,\la,
\varepsilon)-\int\limits_{\varepsilon-2}^{2-\varepsilon}
f(a) \left(\ee_{{\mathcal S}(a)}c_{\vec{m}}\right) da  \right|^2=0
\end{equation}
\end{proposition}

Proof.
We begin by estimating
\begin{equation}
\label{riemannsum}
{\ee}_{\pln} \left| S_n({\vec m}, f,\la,
\varepsilon)-
\frac 1{\sqrt{n}}\sum\limits_{k: |k|\leq (2-\varepsilon)\sqrt{n}}
f\left(\frac k{\sqrt{n}}\right) \left(\ee_{{\mathcal S}(\frac k{\sqrt{n}})}c_{\vec{m}}\right) \right|^2.
\end{equation}

By Lemma \ref{correlplan} and Proposition \ref{ejtheta},  boundedness of $f$ implies the estimate
$$
\left|f\left(\frac k{\sqrt{n}}\right)f\left(\frac l{\sqrt{n}}\right)
{\ee}_{\pln}\left((c_{k+{\vec m}}-
\ee_{{\mathcal S}(\frac{k}{\sqrt{n}})}c_{\vec{m}})\cdot
(c_{l+{\vec m}}-\ee_{{\mathcal S}(\frac{l}{\sqrt{n}})}c_{\vec{m}})\right)\right|\leq \frac{C(f, {\varepsilon})}
{|k-l|+1}.
$$
Summing in $k$ and $l$, we obtain that the integral (\ref{riemannsum}) is bounded above
by  $C(f, \varepsilon)\log^2 n/\sqrt{n}$. Now observe that the quantity
$$
\frac 1{\sqrt{n}}\sum\limits_{k: |k|\leq (2-\varepsilon)\sqrt{n}}
f\left(\frac k{\sqrt{n}}\right) \left(\ee_{{\mathcal S}(\frac{k}{\sqrt{n}})}c_{\vec{m}}\right)
$$
is a Riemann sum for the integral
$$
\int\limits_{\varepsilon-2}^{2-\varepsilon}
f(a) \left(\ee_{{\mathcal S}(a)}c_{\vec{m}}\right) da.
$$
Since the function $f(a)\left(\ee_{{\mathcal S}(a)}c_{\vec{m}}\right)$
is continuous on $[\varepsilon-2, 2-\varepsilon]$, the Riemann sums converge to the integral, and
Proposition \ref{epslocal} is proved.

To derive Lemma \ref{localergodic} from Proposition \ref{epslocal},
note that if $\la\in \Y_{n}$ satisfies $\la_1<(2+\varepsilon_1)\sqrt{n}$, $\la_1^{\prime}<(2+\varepsilon_1)\sqrt{n}$,
then we have
\begin{equation}
\label{epstozero}
|S_n({\vec m},f,\la)-S_n({\vec m},f,\la, \varepsilon)|\leq C({\vec m}, f)(\varepsilon+\varepsilon_1),
\end{equation}
where  $C({\vec m}, f)$ is a constant depending only on ${\vec m}$ and  $f$.
Since $\varepsilon$ and $\varepsilon_1$ can be chosen arbitrarily small,
Lemma \ref{localergodic} follows now from Proposition \ref{ydelta} .

\subsection{Proof of Lemma \ref{hooklength}.}
Observe the clear identity
\begin{equation}
h_k(\la)=\sum\limits_{i=-\infty}^{\infty} \left(c_i(\la)-c_i(\la)c_{i-k}(\la)\right)
\end{equation}
(note that only finitely many terms in the right-hand side are nonzero).
 Lemma \ref{hooklength} is now immediate from Lemma \ref{localergodic}.
For the constant, compute
\begin{multline}
\int\limits_{-2}^2 \left(\frac1\pi \arccos\left(\frac{a}2\right)- \frac1{\pi^2} \arccos^2\left(\frac{a}2\right)+\frac1{k^2\pi^2} \sin^2\left(k\arccos\left(\frac{a}2\right)\right)\right)da=\\
=\frac{32k^2}{(4k^2-1)\pi^2}.
\end{multline}

\subsection{Proof of Lemmas \ref{lochone}, \ref{loch}.}

 For Lemma \ref{loch}, take $\varepsilon>0$ and observe that for $t$ and $h$ satisfying
 $$
 \frac{|t|}{\sqrt{n}}<2-\varepsilon, \ 0< h\leq h_0
 $$
 we have
\begin{equation}
\left| \frac{\sqrt{n}}{h}\left(\Omega\left(\frac{t+h}{\sqrt{n}}\right) - \Omega\left(\frac{t}{\sqrt{n}}\right)\right)-\frac 2{\pi} \arcsin\left(\frac{t}{2\sqrt{n}}\right) \right|\leq \frac{C(\varepsilon, h_0)}{\sqrt{n}},
\end{equation}
where the constant $C(\varepsilon, h_0)$ only depends on $\varepsilon$ and $h_0$.

Now take $s\in (0,1)$ and consider the expression
\begin{equation}
\label{epsphila}
\frac 1{{\sqrt{n}}} \int\limits_0^{h_0}\left(\sum\limits_{k:|k|\leq (2-\varepsilon)\sqrt{n}} \left(\frac{\Phi_{\la}(s+k+h)-\Phi_{\la}(s+k)}{h} -\frac 2{\pi} \arcsin\left(\frac{s+k}{2\sqrt{n}}\right)\right)^2
 \right)dh
\end{equation}

For any $h_0\geq 0$ there exists $N_0=N_0(h_0)$, and for all $h: 0\leq h\leq h_0$,
numbers $\alpha_i^{(h)}$, $i\in {\mathbb Z}$,
satisfying $|\alpha_i^{(h)}|\leq 1$ and $\alpha_i^{(h)}=0$ for  $|i|>N_0(h_0)$ such that for
any $\omega\in\Omega_2$ we have
$$
\frac{\Phi_{\omega}(s+k+h)-\Phi_{\omega}(s+k)}{h}=\sum_{i=-\infty}^{\infty} \alpha_i^{(h)}\omega_{i+k}.
$$

By Proposition \ref{epslocal}, the sum (\ref{epsphila}) converges, with respect to the Plancherel measure,
to the constant
$$
\int\limits_0^{h_0}\int\limits_{\varepsilon-2}^{2-\varepsilon} \ee_{{\mathcal S}(a)}\left(\frac{\Phi_{\omega}(s+h)-\Phi_{\omega}(s)}{h}-
\frac 2{\pi}\arcsin(a/2)\right)^2 da dh.
$$
Taking $\varepsilon$ to $0$ (the transition to the limit is justified in the same way as in (\ref{epstozero}))
and integrating in $s$ from $0$ to $1$, we obtain Lemma \ref{loch}.

\section{Proof of Lemma \ref{tailh}.}
\label{prooftailh}

\subsection{Outline of the Proof.}

The first step in proving Lemma \ref{tailh} is  to reduce integrals
to sums and to observe that summation need only take place ``away from the edge''.

More precisely, let $\delta \in \mathbb R, \, 0 < \delta < \frac{1}{4}$ and let $K>0$.
Denote
\[
\mathbb Y_n(K,\delta) = \{ \lambda \in \mathbb Y_n: \lambda_1 \le 2\sqrt{n} + Kn^\delta, \, \lambda'_1 \le 2\sqrt{n} + Kn^\delta \}.
\]
Denote
\[
F_\lambda^{(L,\delta)}(t) =
\begin{cases}
F_\lambda(t), &\text{if} \quad |t|\le 2\sqrt{n} - Ln^\delta \\
0, &\text{otherwise.}
\end{cases}
\]

\begin{lemma} \label{discrcut}
For any $\delta$ satisfying $0<\delta<\frac{1}{4}$, any $K>0,\, L>0$ and any $\varepsilon >0$, there exists a number $h_0>1$ depending only on $\delta, K, L, \varepsilon$ and such that for any $n \in \mathbb N$ and any $\lambda \in \mathbb Y_n(K,\delta)$ we have the inequality
\begin{multline}
\frac{1}{\sqrt{n}} \int\limits_{h_0}^{+\infty} \int\limits_{-\infty}^{+\infty} \left( \frac{F_\lambda(t+h) - F_\lambda(t)}{h} \right)^2 dt\,dh \leq\\
\leq \frac{4}{\sqrt{n}} \; \sum\limits_{l>h_0-1} \; \sum_{k=-\infty}^{+\infty} \left( \frac{F_\lambda^{(L,\delta)}(k+l) - F_\lambda^{(L,\delta)}(k)}{l} \right)^2 + \varepsilon.
\end{multline}
\end{lemma}
We postpone its proof to the following subsection.

The second step is to estimate the expectation of the quantity
$$
\sum\limits_{l>h_0-1}\sum_{k=-\infty}^{+\infty} \left( \frac{F_\lambda^{(L,\delta)}(k+l) - F_\lambda^{(L,\delta)}(k)}{l} \right)^2
$$
with respect to the Plancherel measure.

We start with estimates for the poissonized Plancherel measure, the Bessel point process.
Take $\delta>1/6$ and let
$$
{\mathcal N}(n, \delta)=\{x\in {\mathbb Z}: |x|\leq 2\sqrt{n}-n^{\delta}\}.
$$

\begin{lemma} \label{varest}
For any $\delta> \frac{1}{6}$ there exist constants $C>0,\, \gamma>0, \, \varepsilon>0$ such that the following holds.

For any $l_0>1$ there exists $n_0>0$ such that for all $n> n_0$, and all $\theta$ satisfying
\[
\left| \frac{\theta}{\sqrt{n}} - 1 \right| < \varepsilon
\]
we have
\begin{equation}
\label{varlogest}
\frac{1}{\sqrt{n}}\, \sum_{l>l_0} \sum_{\substack{k\, \in \, \mathcal N(n,\delta)\\ k+l \,\in\, \mathcal N(n,\delta)}} \left| \frac{\Var_{J(\theta^2)} (c_k + \ldots c_{k+l-1})}{l^2} \right|
\;\le\;
\frac{C \log l_0}{l_0} \, e^{\gamma \left| \theta - \sqrt{n} \right|}.
\end{equation}
\end{lemma}

Lemma \ref{varest} is again essentially a straightforward computation using simple
estimates on the discrete Bessel kernel. We prove Lemma \ref{varest} in the following subsection.
Now we conclude the proof of Lemma \ref{tailh}.

Rewrite formula (\ref{varlogest}) as follows
\begin{multline} \label{cjest}
\frac{1}{\sqrt{n}} \;\sum_{l>l_0} \sum_{\substack{k \,\in \, \mathcal N (n,\delta)\\ k+l \,\in\, \mathcal N(n,\delta)}}
\left| \frac{1}{l} ( \mathbb E_{J(\theta^2)} \left( \sum_{r=k}^{k+l-1} (c_r - J(r,r;\, \theta^2))\right)^2 \right|
\le \\
\le \frac{C\log l_0}{l_0} \,e^{\gamma \left| \theta - \sqrt{n} \right|}.
\end{multline}
Now write
\begin{multline} \label{FBessel}
F_\lambda (k+1) - F_\lambda(k) = 1 - 2c_k(\lambda) - \sqrt{n} \left( \Omega \left( \frac{k+1}{\sqrt{n}} \right) - \Omega \left( \frac{k}{\sqrt{n}} \right) \right) =\\= 2 \left( \frac{\arccos{\frac{k}{2\sqrt{n}}}}{\pi} - c_k(\lambda) \right) + \frac{2}{\pi} \arcsin\left( \frac{k}{2\sqrt{n}} \right) - \sqrt{n} \left( \Omega \left( \frac{k+1}{\sqrt{n}} \right) - \Omega \left( \frac{k}{\sqrt{n}} \right) \right)=\\=
2\left( J(k,k; \, \theta^2) - c_k (\lambda) \right) + 2\left( \frac{\arccos \frac{k}{2\sqrt{n}}}{\pi} - J(k,k; \, \theta^2)\right)+\\
+
\left( \, \frac{2}{\pi} \arcsin\left( \frac{k}{2\sqrt{n}} \right) - \sqrt{n} \left( \Omega \left( \frac{k+1}{\sqrt{n}} \right) - \Omega \left( \frac{k}{\sqrt{n}} \right) \right) \right).
\end{multline}
From the Taylor formula applied to the function $\Omega$ we have, for $\left| k \right| < 2 \sqrt{n}$, the estimate
\begin{equation} \label{omarc}
\left| \sqrt{n} \left( \Omega \left( \frac{k+1}{\sqrt{n}} \right) - \Omega \left( \frac{k}{\sqrt{n}} \right) \right)-
\frac{2}{\pi} \arcsin\left( \frac{k}{2\sqrt{n}} \right) \right| \le \frac{10}{\sqrt{4n-k^2}}.
\end{equation}
To estimate the quantity
\[
\left| J(k,k; \, \theta^2) - \frac{\arccos \frac{k}{2\sqrt{n}}}{\pi} \right|
\]
we use the following Lemma.

\begin{lemma} \label{besselmain}
There exists $\varepsilon_0>0$ such that the following holds. For any $\delta_0 > \frac{1}{6}$ there exist constants $C>0,\, \gamma>0$ such that for all $n \in \mathbb N$, all $x \in \mathbb Z$ satisfying $\left| x\right| \le 2 \sqrt{n} - n^{\delta_0}$ and all $\theta \in \mathbb C$ satisfying
\[
\left| \frac{\theta}{\sqrt{n}} - 1 \right| < \varepsilon_0
\]
we have
\[
\left| J(x,x; \, \theta^2) - \frac{1}{\pi} \arccos \frac{x}{2\sqrt{n}} \right| \le
\frac{C}{2\sqrt{n} - \left| x \right|} \,e^{\gamma \left| \theta - \sqrt{n} \right| }.
\]
\end{lemma}
Observe that it suffices to prove Lemma \ref{besselmain} for $x>0$, as the other case follows by symmetry. Lemma \ref{besselmain} is again a relatively straightforward estimate using Okounkov's contour integral representation for the discrete Bessel kernel. For the reader's convenience, we give the proof in the last Section.

Using Lemma \ref{besselmain}, we obtain from (\ref{FBessel}), (\ref{omarc}) the estimate
\[
\left| \left( F_\lambda(k+1) - F_\lambda (k) \right) - 2 \left( J(k,k;\, \theta^2) - c_k(\lambda) \right) \right| \;\le\;
\frac{C}{2\sqrt{n} - \left| k \right|} \,e^{\gamma \left| \theta - \sqrt{n} \right|}.
\]
Observe now the following simple inequality
\begin{equation} \label{minusx}
\frac{1}{\sqrt{n}} \;\sum_{l>l_0} \sum\limits_{\substack{k \,\in \, \mathcal N (n,\delta)\\ k+l \,\in\, \mathcal N(n,\delta)}}
\left( \frac{\sum\limits_{x=k}^{k+l} \; \frac{1}{2\sqrt{n} - \left| x\right|} }{l} \right)^2 \;\le\; \frac{C \log^2 l_0}{l_0}.
\end{equation}
From (\ref{cjest}) and (\ref{minusx}) we now obtain
\[
\frac{1}{\sqrt{n}} \;\sum_{l>l_0} \sum\limits_{\substack{k \,\in \, \mathcal N (n,\delta)\\ k+l \,\in\, \mathcal N(n,\delta)}}
\left|\, \mathbb E_{J(\theta^2)} \left(
\frac{F_\lambda(k+l) - F_\lambda(k)}{l} \right)^2\, \right| \le
\frac{C \log^2 l_0}{l_0}\, e^{\gamma \left| \theta - \sqrt{n} \right|}.
\]
Depoissonizing, we have
\begin{equation}
\label{planlog}
\frac{1}{\sqrt{n}} \mathbb \, E_{{\mathbb Pl}^{(n)}} \left(
 \;\sum_{l>l_0} \sum\limits_{\substack{k \,\in \, \mathcal N (n,\delta)\\ k+l \,\in\, \mathcal N(n,\delta)}}
\left( \frac{F_\lambda(k+l) - F_\lambda(k)}{l} \right)^2\, \right) \le
\frac{C \log^2 l_0}{l_0}.
\end{equation}

The estimate (\ref{planlog}), together with Lemma \ref{discrcut}, completes the proof of Lemma \ref{tailh}.

\subsection{Proof of Lemma \ref{discrcut}}
The first step is to pass from integrals in $t$ and in $h$ to sums in $k$ and $l$. Since the function $F_\lambda$ is Lipschitz with the Lipschitz constant $2$, for any $t \in \mathbb R, \, h \in \mathbb R_+$ we have
\[
\left( F_\lambda (t+h) - F_\lambda (t) \right)^2 \le 2 \left( F_\lambda (k+l) - F_\lambda(k) \right)^2 + 16,
\]
where $k=[t], \, l=[h]$. Integrating in $t$, we have
\[
\int\limits_{-\infty}^{+\infty} \left( F_\lambda (t+h) - F_\lambda (t) \right)^2 \, dt \le 2 \sum\limits_{k=-\infty}^{+\infty} \left( F_\lambda (k+l) - F_\lambda(k) \right)^2 + 40 \sqrt{n},
\]
where, as before $l = [h]$.

\noindent
Now, integrating in $h$ from $l$ to $l+1$, we arrive at the inequality
\[
\int\limits_{l}^{l+1} \int\limits_{-\infty}^{+\infty} \left( F_\lambda(t+h) - F_\lambda(t) \right)^2 dt\, dh \le 2 \sum\limits_{k=-\infty}^{+\infty} \left( F_\lambda (k+l) - F_\lambda(k) \right)^2 + 40\sqrt{n},
\]
whence, for any $h_0>1$ we have
\begin{multline}
\frac{1}{\sqrt{n}} \int\limits_{h_0}^{+\infty} \int\limits_{-\infty}^{+\infty} \left( \frac{F_\lambda (t+h) - F_\lambda (t)}{h} \right)^2 dt\, dh \le\\
\le \frac{2}{\sqrt{n}} \, \sum\limits_{l>h_0-1} \; \sum_{k=-\infty}^{+\infty} \left( \frac{F_\lambda (k+l) - F_\lambda (k)}{l} \right)^2 + \; 40\cdot \sum\limits_{l>h_0-1} \frac{1}{l^2}.
\end{multline}
For $h_0>1$, we thus arrive at the inequality
\begin{multline}
\frac{1}{\sqrt{n}} \int\limits_{h_0}^{+\infty} \int\limits_{-\infty}^{+\infty} \left( \frac{F_\lambda (t+h) - F_\lambda (t)}{h} \right)^2 dt\, dh \le\\
\le \frac{2}{\sqrt{n}} \, \sum\limits_{l>h_0-1} \; \sum_{k=-\infty}^{+\infty} \left( \frac{F_\lambda (k+l) - F_\lambda (k)}{l} \right)^2 + \; 40 \left( \frac{1}{h_0-1} + \frac{1}{(h_0-1)^2} \right),
\end{multline}
which concludes the first step of the argument.

To prove Lemma \ref{discrcut}, it suffices now to establish the following Lemma.

\begin{lemma} For any $\delta$ satisfying $0<\delta<\frac{1}{4}$, any $K,L >0$ there exists a positive constant $C(K,L,\delta)$ such that for any $n \in \mathbb N$ any $\lambda \in \mathbb Y_n (K,\delta)$ and any $h \ge 1$ we have:
\begin{multline*}
\frac{2}{\sqrt{n}} \, \sum\limits_{l\ge h} \; \sum_{k=-\infty}^{+\infty} \left( \frac{F_\lambda (k+l) - F_\lambda (k)}{l} \right)^2 \le\\
\le \frac{2}{\sqrt{n}} \, \sum\limits_{l\ge h} \; \sum_{k=-\infty}^{+\infty} \left( \frac{F_\lambda^{(L,\delta)} (k+l) - F_\lambda^{(L,\delta)} (k)}{l} \right)^2
+C(K,L,\delta) \cdot n^{2\delta - \frac{1}{2}}.
\end{multline*}
\end{lemma}

The next step is to pass from $F_\lambda$ to $F_\lambda^{(L,\delta)}$. Assume $\lambda \in \mathbb Y_n (K,\delta)$. Denote
\[
\check F_\lambda^{(L,\delta)} (t) = F_\lambda (t) - F_\lambda^{(L,\delta)} (t).
\]
The support of $\check F_\lambda^{(L,\delta)}$ consists of two intervals, each of length at most $(K+L)\cdot n^\delta$. Write
\[
\check I_\lambda^{(L,\delta)} = \{ t\, : \, \check F_\lambda^{(L,\delta)} (t) \ne 0 \},
\]
and, for $l \in \mathbb N$, denote
\[
\check I_{\lambda, l}^{(L,\delta)} = \{ k \in \mathbb Z: [k,k+l] \cap \check I_\lambda^{(L,\delta)} \ne \emptyset \}.
\]
By definition, for the cardinality of $\check I_{\lambda, l}^{(L,\delta)}$ we have:
\[
\#\check I_{\lambda, l}^{(L,\delta)} \le 2(2l + (K+L)n^\delta),
\]
whence, using the clear inequality
\[
\left( F_\lambda (k+l) - F_\lambda (k) \right)^2 \le 4l^2,
\]
we arrive, for any $l \in \mathbb N$, at the inequality
\[
\sum\limits_{k \in \check I_{\lambda, l}^{(L,\delta)}} \left( F_\lambda (k+l) - F_\lambda (k) \right)^2 \le 8l^2 \left( 2l + (K+L)n^\delta \right),
\]
and, consequently, at the inequality
\begin{multline} \label{lsmall}
\sum\limits_{k=-\infty}^{+\infty} \left( F_\lambda (k+l) - F_\lambda (k) \right)^2 \le
\sum\limits_{k=-\infty}^{+\infty} \left( F_\lambda^{(L,\delta)} (k+l) - F_\lambda^{(L,\delta)} (k) \right)^2 +\\
+ 8l^2 \left( 2l + (K+L) n^\delta \right).
\end{multline}
Furthermore, again using the fact that $\check I_\lambda^{(L,\delta)}$ has measure at most $2(K+L)n^\delta$, and the Lipschitz property of $\check F_\lambda^{(L,\delta)}$ for any $l \in \mathbb N$ we have the inequality
\[
\sum\limits_{k=-\infty}^{+\infty} \left( \check F_\lambda^{(L,\delta)} (k+l) - \check F_\lambda^{(L,\delta)} (k) \right)^2 \le 64\cdot (K+L)^3 \cdot n^{3\delta}.
\]
Using the clear inequality
\begin{multline}
\sum\limits_{k=-\infty}^{+\infty} \left( F_\lambda (k+l) - F_\lambda (k) \right)^2 \le \\
\le 2 \sum\limits_{k=-\infty}^{+\infty} \left( F_\lambda^{(L,\delta)} (k+l) - F_\lambda^{(L,\delta)} (k) \right)^2 +
2 \sum\limits_{k=-\infty}^{+\infty} \left( \check F_\lambda^{(L,\delta)} (k+l) - \check F_\lambda^{(L,\delta)} (k) \right)^2
\end{multline}
we obtain, for any $l \in \mathbb N$, the inequality
\begin{multline} \label{lbig}
\sum\limits_{k=-\infty}^{+\infty} \left( F_\lambda (k+l) - F_\lambda (k) \right)^2 \le\\
\le 2 \sum\limits_{k=-\infty}^{+\infty} \left( F_\lambda^{(L,\delta)} (k+l) - F_\lambda^{(L,\delta)}(k) \right)^2 + 128 \cdot (K+L)^3 \cdot n^{3\delta}.
\end{multline}

We proceed to summing in $l$. First we sum in $l\in [h_0, n^\delta]$ using inequality (\ref{lsmall}), and then we sum in $l \in [n^\delta, +\infty)$ using inequality $(\ref{lbig})$. From (\ref{lsmall}) we immediately obtain
\begin{multline} \label{ltondelta}
\sum\limits_{h_0 \le l \le n^\delta} \sum\limits_{k=-\infty}^{\infty}
\left( \frac{F_\lambda (k+l) - F_\lambda (k)}{l} \right)^2\le\\
\le \sum\limits_{h_0 \le l \le n^\delta} \sum\limits_{k=-\infty}^{\infty}\left( \frac{F_\lambda^{(L,\delta)} (k+l) - F_\lambda^{(L,\delta)} (k)}{l} \right)^2 +
8\cdot (K+L+2)\cdot n^{2\delta}.
\end{multline}
From (\ref{lbig}) we have
\begin{multline} \label{ltoinfty}
\sum\limits_{l>n^\delta}\sum\limits_{k=-\infty}^{\infty} \left( \frac{F_\lambda (k+l) - F_\lambda (k)}{l} \right)^2 \le \\
\le 2 \sum\limits_{l>n^\delta}\sum\limits_{k=-\infty}^{\infty} \left( \frac{F_\lambda^{(L,\delta)} (k+l) - F_\lambda^{(L,\delta)}(k)}{l} \right)^2 +
+128\cdot (K+L)^3 \cdot \left( n^{2\delta} + n^\delta \right).
\end{multline}
Indeed, to prove (\ref{ltoinfty}), it suffices to observe that
\[
\sum\limits_{l>n^\delta} \frac{1}{l^2} < \frac{1}{n^{2\delta}} + \frac{1}{n^\delta}.
\]
Combining (\ref{ltondelta}) and (\ref{ltoinfty}), we conclude the proof of the Lemma.

\subsection{Average Variance of the Discrete Bessel Process.}

Let $l_0>1$ and for $r \in \mathbb N$ denote
\[
\psi (l_0,r) =
\begin{cases}
2 \left( \frac{1}{l_0} + \ldots + \frac{1}{r-1} + r \left( \sum\limits_{n=r}^{+\infty}\frac{1}{n^2}\right) \right), &\text{if} \quad r>l_0 \\
2r \left( \sum\limits_{n=l_0}^{+\infty} \frac{1}{n^2} \right), &\text{if} \quad r\leq l_0
\end{cases}
\]

\begin{proposition}
For any $\theta \in \mathbb C$ we have
\[
\sum\limits_{l=l_0}^{+\infty} \sum\limits_{k=-\infty}^{+\infty} \frac{\Var_{J(\theta^2)} (c_k + \ldots + c_{k+l-1})}{l^2} = \sum_{r=1}^{+\infty} \psi (l_0,r) \cdot
\left( \; \sum\limits_{\substack{x,y \in \mathbb Z \\ y-x=r}} \left( J(x,y; \theta^2) \right)^2 \right).
\]
\end{proposition}

\begin{proof}
From the well-known identity
\[
J(x,x; \, \theta^2) = \sum\limits_{y \in \mathbb Z} \left( J(x,y; \theta^2) \right)^2
\]
we derive
\[
\Var_{J(\theta^2)} (c_k + \ldots + c_{k+l-1} ) = \sum\limits_{\substack{x \in \mathbb Z \\ x \in [k, k+l]}} \; \sum\limits_{\substack{ y \in \mathbb Z \\ y \notin [k,k+l]}} \; \left( J(x,y; \theta^2) \right)^2.
\]
Summing in $k$, we obtain
\[
\sum\limits_{k=-\infty}^{+\infty} \Var_{J(\theta^2)} (c_k + \ldots + c_{k+l-1}) = 2 \sum\limits_{\substack{x,y \in \mathbb Z \\ x<y}} \min (y-x, l) \cdot \left( J(x,y; \theta^2) \right)^2.
\]
Dividing by $l^2$ and summing in $l$ we obtain the Proposition.
\end{proof}
Recall that ${\mathcal N}(n, \delta)=\{x\in {\mathbb Z}: |x|\leq 2\sqrt{n}-n^{\delta}\}$.

Summing only over indices belonging to  $\mathcal N(n,\delta)$, we obtain
\begin{corollary} \label{varform}  For any $\delta>0,\, l>1$ we have
\[
\sum_{l>l_0} \sum_{\substack{k\in \mathcal N(n,\delta)\\ k+l \,\in\, \mathcal N(n,\delta)}} \left| \frac{\Var_{J(\theta^2)} (c_k + \ldots c_{k+l-1})}{l^2} \right|
\;\le\;  \sum\limits_{r=1}^{+\infty} \psi(l_0,r)
\; \cdot
\sum\limits_{\substack{ x \,\in\, \mathcal N(n,\delta), y\in {\mathbb Z} \\ \left| x-y \right| = r}} \left| J(x,y;\, \theta^2) \right|^2
\]
\end{corollary}

\subsection{Estimates on the Bessel Kernel.}

To estimate the average variance of the discrete Bessel process using Corollary \ref{varform},
we need estimates for the discrete Bessel kernel for various values of the parameters.
We formulate these estimates in this subsection and postpone their routine proofs until the last Section.
Very simple  estimates on the Bessel kernel are quite sufficient for our purposes.

We start with the following estimate for the Bessel function.
\begin{lemma}
\label{bescontour}
There exists $\varepsilon_0 > 0$ such that the following holds.
For any $\delta > \frac{1}{6}$ there exist constants $C>0,\, \gamma>0$ depending only on $\delta$ such that for all
$x \in \mathbb N$ satisfying
\[
0\leq x \le 2\sqrt{n} - n^\delta
\]
and all $\theta \in \mathbb C$ satisfying $\left| \frac{\theta}{\sqrt{n}} - 1 \right| < \varepsilon_0$ we have
\begin{enumerate}
\item $\big| J_x (2\theta) \big| \le \frac{C e^{\gamma \left| \theta - \sqrt{n}\right|}}{n^{1/8}\sqrt[4]{{2\sqrt{n}-x}}}$;
\item $\big| J_{x+1} (2 \theta) - J_x (2\theta) \big| \le \frac{C e^{\gamma \left| \theta - \sqrt{n}\right|}\sqrt[4]{{2\sqrt{n}-x}}}{n^{{3}/{8}}}$
\end{enumerate}
\end{lemma}

{\bf Remark.} By symmetry, for $x<0$,  $|x| \le 2\sqrt{n} - n^\delta$, we have
$$\big| J_x (2\theta) \big| \le \frac{C e^{\gamma \left| \theta - \sqrt{n}\right|}}{n^{{1}/{8}}\sqrt[4]
{2\sqrt{n}-|x|}};
$$
$$\big| J_{x+1} (2 \theta) + J_x (2\theta) \big| \le \frac{C e^{\gamma \left| \theta - \sqrt{n}\right|}\sqrt[4]{{2\sqrt{n}-|x|}}}{n^{{3}/{8}}}.$$

This is easily proved using the contour integral for the Bessel function; for the reader's convenience, the proof is given in the last Section.

\begin{corollary} \label{bescan} Let $\delta> \frac{1}{6}$. There exist constants $C>0,\, \gamma>0, \, \varepsilon>0$ such that for all $n \in \mathbb N$, all $\theta \in \mathbb C$ satisfying
\[
\left| \frac{\theta}{\sqrt{n}} - 1 \right| < \varepsilon
\]
and all $x,y\in \mathbb N,\, x\ne y, x,y \in \mathcal N(n,\delta)$ we have
\[
\left| J(x,y;\, \theta^2) \right|
\;\le\;
\frac{C}{\left| x - y\right|} \left( \sqrt[4]{\frac{2\sqrt{n}-x}{2\sqrt{n}-y}} + \sqrt[4]{\frac{2\sqrt{n}-y}{2\sqrt{n}-x}}\right) \cdot e^{\gamma \left| \theta - \sqrt{n}\right|}.
\]
\end{corollary}
\begin{proof}
Immediate from the preceding Lemma and the formula
\begin{multline}
J_{x+1}(2\theta) J_y (2 \theta) - J_{y+1} (2\theta) J_x (2\theta) =\\
=\left( J_{x+1} (2 \theta) - J_x (2\theta) \right) J_y(2\theta)
-\left( J_{y+1}(2\theta) - J_y(2\theta) \right) J_x(2\theta).
\end{multline}
\end{proof}

\begin{corollary} \label{bestriv} Let $\delta > \frac{1}{6}$. There exist constants $C>0,\, \gamma>0,\, \varepsilon >0$ such that for all $n \in \mathbb N$, all $\theta \in \mathbb C$ satisfying
\[
\left| \frac{\theta}{\sqrt{n}} - 1 \right| < \varepsilon
\]
and all $x,y\in \mathbb Z,\, x\ne y, x,y \in \mathcal N(n,\delta)$ we have
\[
\left| J(x,y; \, \theta^2) \right|
\;\le\;
C \cdot \frac{\sqrt[4]{n} \cdot e^{\gamma \left| \theta - \sqrt{n} \right|}}{\sqrt[4]{(2\sqrt{n} - \left| x\right|)(2\sqrt{n} - \left| y \right|)}\left| x-y\right|}.
\]
\end{corollary}

These are immediate from Lemma \ref{bescontour}.

For the values of $x$ that are ``close the edge'', we use the following estimate.

\begin{proposition} \label{onesix} There exist constants $C>0,\, \varepsilon>0, \, \gamma>0$ such that for all $x \in \mathbb Z$, all $n \in \mathbb N$ and all $\theta \in \mathbb C$ satisfying
\[
\left| \frac{\theta}{\sqrt{n}} - 1 \right| < \varepsilon
\]
we have
\[
\left| J_x (2 \theta) \right| \;\le\; \frac{C e^{\gamma \left| \theta - \sqrt{n} \right|}}{n^{\frac{1}{6}}}.
\]
\end{proposition}

Proposition \ref{onesix} is immediate from the uniform asymptotic estimates of Olver for Bessel functions
(see, e.g., Abramowitz and Stegun \cite{AS}, 9.3.5, 9.3.6, 9.3.38, 9.3.39).

For values of $x$ ``beyond the edge'', we use the following stretched exponential estimate.

\begin{proposition} \label{expest}
There exists  constants $\varepsilon_1>0, \varepsilon_2>0$ such that the following is true.
For any $\delta > \frac{1}{6}$ there exist constants $C>0, \, \varepsilon>0, \, \gamma>0, \, \widetilde\gamma>0$ such that for all $\delta: \, \delta_0 \le \delta \le \frac{1}{2}$, all $n \in  \mathbb N$, all $\theta \in \mathbb C$ satisfying $\left| \frac{\theta}{\sqrt{n}} - 1 \right| < \varepsilon$ and all $x \in \mathbb N$ satisfying
$2\sqrt{n} + n^\delta<x<(2+\varepsilon_1)\sqrt{n}$, we have
\[
\left| J_x(2\theta) \right|
\;\le\;
C \exp\left(-\,\widetilde\gamma \, \frac{(x-2\sqrt{n})^{3/2}}{n^{{1}/{4}}} + \gamma \left| \theta - \sqrt{n} \right|\right),
\]
while for $x>(2+\varepsilon_1)\sqrt{n}$ we have
\[
\left| J_x(2\theta) \right|
\;\le\;
C \exp\left(-\varepsilon_2(x-2\sqrt{n}) + \gamma \left| \theta - \sqrt{n} \right|\right).
\]

\end{proposition}

Proposition \ref{expest} is immediate from the contour integral representation of Bessel functions.
Note that Proposition \ref{ydelta} is immediate from Proposition \ref{expest} by depoissonization.

\subsection{Proof of Lemma \ref{varest}.}

We now derive Lemma \ref{varest} from Corollary \ref{varform}.
\begin{proof}
First, note the clear inequality
\[
\psi(l_0,r) \;<\;
\begin{cases}
2\left( \log\left( \frac{r}{l_0-1} \right) + \frac{r}{r-1} \right) & \text{for } l_0 \le r \\
\frac{2r}{l_0-1} & \text{for } l_0>r
\end{cases}
\]
We further assume that $n$ is large enough so that $n^\delta> l_0$. We consider several cases.

{\it{Case 1. The Bulk.}}

First, take $\varepsilon>0$ and denote
\[
\mathcal N_\varepsilon = \left\{ (x,y) \in \mathbb Z^2: \, \left| x \right|, \left| y \right| \le (2-\varepsilon) \sqrt{n} \, \right\}.
\]
In this case, the Debye asymptotics implies
\[
\left| J\left(x,y;\, \theta^2\right) \right|^2 \;\le\; \frac{Ce^{\gamma \left| \theta - \sqrt{n} \right|}}{\left( \left| x-y\right| +1 \right)^2},
\]
and we consequently have
\[
\frac{1}{\sqrt{n}} \sum_{(x,y) \in \mathcal N_\varepsilon} \psi(l_0,r) \left| J (x,y; \, \theta^2) \right|^2 \;\le\; \frac{C \log l_0}{l_0} e^{\gamma \left| \theta - \sqrt{n} \right|}.
\]
We proceed to the analysis of the remaining terms. Recall our notation
\[
\mathcal N(n, \delta) = \left\{ x \in \mathbb Z: \, \left| x\right| \le 2 \sqrt{n} - n^\delta \right\}, \quad r = \left| x-y\right|
\]
and denote
\[
\mathcal N_2(n,\delta)  = \left\{ y \in \mathbb Z: \left| y\right| \le 2 \sqrt{n} - \frac{1}{2} n^\delta \right\}.
\]

{\it{Case 2. Close Points.}}

Denote
\[
\mathcal N_2 = \left\{ (x,y) \in \mathbb Z^2:\; x \in \mathcal N(n,\delta), y \in \mathcal N_2(n,\delta), \;\left| x-y\right| \le n^\delta \right\}.
\]
In this case, again, we have
\[
\left| J\left(x,y; \, \theta^2\right) \right|^2 \; \le \; \frac{C e^{\gamma \left| \theta - \sqrt{n}\right|}}{\left( \left| x-y\right| +1 \right)^2}.
\]
Indeed, if $x$ and $y$ are both positive, then the estimate follows from Corollary \ref{bescan}; if both are negative,
then the estimate follows by symmetry;
while if $x$ and $y$ have different signs, then $(x,y) \in \mathcal N_\varepsilon$, and the estimate also holds. Consequently, again we have
\[
\frac{1}{\sqrt{n}} \sum_{(x,y) \in \,\mathcal N_2} \psi(l_0,r) \left| J\left(x,y; \theta^2\right) \right|^2 \; \le \; \frac{C \log l_0}{l_0} e^{\gamma \left| \theta - \sqrt{n} \right|}.
\]

{\bf{Remark.}} Of course, $\mathcal N_\varepsilon \cap \mathcal N_2 \ne \emptyset$, but that does not matter since we are only concerned with upper estimates.

{\it{Case 3. Distant Positive Points.}}

Set
\[
\mathcal N_3 = \left\{ (x,y) \in \mathbb Z^2: \, x\ge 0, \, y\ge 0, \, x \in \mathcal N(n,\delta), \, y \in \mathcal N_2(n,\delta), \, \left| x-y\right| > n^\delta \right\}.
\]
For definiteness, let $x>y$. Corollary \ref{bescan} implies
the bound
\[
\left| J\left(x,y; \, \theta^2\right) \right|^2 \; \le \; \frac{C e^{\gamma \left| \theta - \sqrt{n} \right|}}{r^2} \left( 1 + \sqrt{\frac{r}{2\sqrt{n} - x}} \,\right)
\]
and, consequently, there exists $\delta_2 >0$ depending only on $\delta$ such that
\[
\frac{1}{\sqrt{n}} \sum_{(x,y) \in {\mathcal N}_3} \psi(l_0, r) \cdot \left| J\left(x,y;\, \theta^2\right) \right|^2 \;\le\; Cn^{-\delta_2}\cdot e^{\gamma \left| \theta - \sqrt{n} \right|}.
\]

{\it{Case 4. Distant Negative Points.}}
Set
\[
\mathcal N_4 = \left\{ (x,y) \in \mathbb Z^2: \, x\le 0, \, y\le 0, \, x \in \mathcal N(n,\delta), \, y \in \mathcal N_2(n,\delta), \, \left| x-y\right| > n^\delta \right\}.
\]
This case is similar to the previous one.

{\it Case 5. Distant Points of Opposite Signs.}

Set
\[
\mathcal N_5 = \left\{ (x,y) \in \mathbb Z^2: \, x \in \mathcal N(n,\delta), \, y \in \mathcal N_2(n,\delta), \, (x,y) \notin \mathcal N_\varepsilon,\; x \text{ and } y \text{ have opposite signs} \right\}.
\]
In this case $\left| x-y \right| \ge \sqrt{n}$, and Corollary \ref{bestriv} implies
\[
\left| J\left(x,y; \theta^2 \right) \right| \;\le\;
\frac{C e^{\gamma \left| \theta - \sqrt{n} \right|}}{\sqrt{n} \cdot \sqrt[4]{\left( 2\sqrt{n} - \left| x\right| \right) \left( 2 \sqrt{n} - \left| y\right| \right)}},
\]
whence
\[
\frac{1}{\sqrt{n}} \sum_{(x,y) \in \, \mathcal N_4} \psi(l_0,r) \cdot \left| J \left( x,y; \, \theta^2\right) \right|^2 \; \le \; \frac{C e^{\gamma \left| \theta - \sqrt{n} \right|}}{\sqrt{n}}.
\]

{\it{Case 6. Large Values of $y$.}}

Set
\[
\mathcal N_6 = \left\{ (x,y) \in \mathbb Z^2: \, x \in \mathcal N(n,\delta), \; \left| y \right| \ge 2\sqrt{n} + n^\delta \right\}.
\]
In this case, by Proposition \ref{expest}, there exists $\delta_6 > 0$ depending only on $\delta$ such that
\[
\left| J \left( x,y; \, \theta^2 \right) \right| \; \le \; C e^{\gamma \left| \theta - \sqrt{n} \right| - (y - 2\sqrt{n} )^{\delta_6}}
\]
\begin{center}
\end{center}
and, consequently, there exists $\delta_7>0$ depending only on $\delta_6$ such that we have
\[
\sum_{(x,y) \in \, \mathcal N_6} \psi (l_0, r) \cdot \left| J\left(x,y; \, \theta^2 \right) \right|^2 \;\le\;
C e^{\gamma \left| \theta - \sqrt{n} \right| - n^{\delta_7}}.
\]

{\it{Case 7. The Point $y$ on the Edge.}}

Set
\[
\mathcal N_7 = \left\{ (x,y) \in \mathbb Z^2: \, x \in \mathcal N(n,\delta), \; 2\sqrt{n} - \frac{1}{2}n^\delta \le \left| y \right| \le 2\sqrt{n} + n^\delta \right\}.
\]
In this case, Lemma \ref{bescontour} and Proposition \ref{onesix} give
\[
\left| J\left( x,y; \, \theta^2 \right) \right| \; \le \;
\frac{C e^{\gamma \left| \theta - \sqrt{n} \right|} \cdot n^{\frac{5}{24}}}{\left( 2\sqrt{n} - \left| x \right|\right)^{\frac{5}{4}}},
\]
whence
\begin{multline*}
\frac{1}{\sqrt{n}} \sum_{(x,y) \in \, \mathcal N_7} \psi(l_0,r) \cdot \left| J \left( x,y; \, \theta^2 \right) \right|^2
\;\le\;\\
\le
C e^{\gamma \left| \theta - \sqrt{n} \right|} \cdot \frac{n^{\frac{5}{12}}}{\sqrt{n}} \cdot n^\delta \cdot \sum_{x \in \, \mathcal N(n,\delta)} \frac{1}{\left(2 \sqrt{n} - \left| x\right| \right)^{\frac{5}{2}}} \le
\\
\le C e^{\gamma \left| \theta - \sqrt{n} \right|} \cdot n^{-\frac{1}{12}-\frac{\delta}{2}}.
\end{multline*}
The Lemma is proved completely.
\end{proof}

\section{Proofs of Estimates for the Discrete Bessel Kernel.}
\label{secbessel}

\subsection{Proof of Lemma \ref{bescontour}}
It is convenient to prove the following equivalent reformulation of Lemma \ref{bescontour}.
\begin{lemma}
\label{bescontourtwo}
There exists $\varepsilon_0 > 0$ such that the following holds.

\noindent
For any $\delta_0 > \frac{1}{6}$ and any $K>0$ there exist constants $C>0,\, \gamma>0$ depending only on $\delta_0$ and $K$, such that for all $\delta: \, \delta_0\le \delta \le \frac{1}{2}$, $n \in \mathbb N$, all $x \in \mathbb N$ satisfying
\[
x \le 2\sqrt{n} - K n^\delta
\]
and all $\theta \in \mathbb C$ satisfying $\left| \frac{\theta}{\sqrt{n}} - 1 \right| < \varepsilon_0$ we have

\begin{enumerate}
 \item $\big| J_x (2\theta) \big| \le \frac{C e^{\gamma \left| \theta - \sqrt{n}\right|}}{n^{\frac{\delta}{4} + \frac{1}{8}}};$
 \item $\big| J_{x+1} (2 \theta) - J_x (2\theta) \big| \le \frac{C e^{\gamma \left| \theta - \sqrt{n}\right|}}{n^{\frac{3}{8} - \frac{\delta}{4}}}$
\end{enumerate}
\end{lemma}
Throughout the proof, the symbols $C$ and $\gamma$ will denote constants depending only on $\delta_0$ and $K$.
Let $\mathcal K$ be a contour going around $0$ counterclockwise once.
We then have the following integral representation for the Bessel function:
\[
J_x (2\theta) = \frac{1}{2\pi i} \oint\limits_{\mathcal K} e^{\theta (z - z^{-1})} \cdot z^{-x-1} \, dz
\]
We choose $\mathcal K$ as follows. Take $\varepsilon > 0$ sufficiently small. Denote $u = \frac{x+1}{\sqrt{n}}$ and introduce the angle $\varphi_u$,  $0\leq
\varphi_u \leq \frac{\pi}{2}$, by the formula
\[
2\cos \varphi_u = u.
\]
Introduce the arcs $I^+,\, I^-$ by the formulas:
\[
 \begin{array}{lc}
I^+ = e^{i \varphi_u} + \frac{1+i}{\sqrt{2}} e^{i \varphi_u} \cdot t, &\qquad |t| \le n^{-\frac{1}{8}-\frac{\delta}{4}+ \varepsilon} \\
\\
I^- = e^{-i \varphi_u} + \frac{1-i}{\sqrt{2}} e^{-i \varphi_u} \cdot t, &\qquad |t| \le n^{-\frac{1}{8}-\frac{\delta}{4}+ \varepsilon}
\end{array}
\]

\begin{figure}[h!]
\centering
\includegraphics{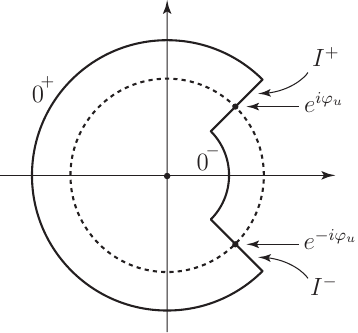}
\caption{}
\label{fig:Fig01}
\end{figure}

We now complete the contour $\mathcal K$ by drawing a circle arc $O^+$ counterclockwise from the outer endpoint of $I^+$ to the outer endpoint of $I^-$, and, similarly, drawing a circle arc $O^-$ counterclockwise from the inner endpoint of $I^-$ to the inner endpoint of $I^+$ (see Fig. \ref{fig:Fig01}). We estimate the contribution of each arc separately and show that the arcs $I^+,\, I^-$ give the main contribution, while the contribution of the arcs $O^+, \, O^-$ is negligible. We start with the arc $I^+$.  Write
\[
e^{\theta(z - z^{-1})} \cdot z^{-x-1} = e^{{\sqrt{n}}(z - z^{-1} - u\log z)}\cdot e^{(\theta-{\sqrt{n}})(z-z^{-1})},
\]
and denote
\[
S(z) = z - z^{-1} - u\log z.
\]
Write
\[
z_t = e^{i\varphi_u} + \frac{1+i}{\sqrt{2}} e^{i\varphi_u} \cdot t
\]
and denote $\widetilde S(t) = S(z_t)$.

Since
\[
\frac{dS}{dz} \Big|_{z=e^{i\varphi_u}} = \,0, \qquad
\frac{d^2 S}{dz^2}\Big|_{z=e^{i\varphi_u}} =\, 2i \sin \varphi_u \cdot e^{-2i\varphi_u}
\]
and since there exists a constant $C$ such that
\[
\left| \frac{d^3 S}{dz^3} \right| \le C \qquad \text{for all} \quad z \in I^+,
\]
one can write
\[
\widetilde S(t) = - \sin \varphi_u \cdot t^2 + A(t) t^3,
\]
where $\left| A(t) \right| \le C$ provided $\left| t \right| \le n^{-\frac{1}{8}- \frac{\delta}{4} + \varepsilon}$.

Consequently,
\[
\int\limits_{I^+} e^{\theta(z - z^{-1})} \cdot z^{-x-1}\, dz = \int\limits_{I^+} e^{(\theta - \sqrt{n})(z_t - z_t^{-1})}\cdot e^{\sqrt{n}\,\widetilde S(t)} \, dt=
\]
\[
= \int\limits_{I^+} e^{(\theta - \sqrt{n})(z_t - z_t^{-1})} \cdot e^{-\sqrt{n} \sin \varphi_u \cdot t^2} \cdot e^{\sqrt{n}\, A(t) \cdot t^3} \, dt.
\]
Noting that $z_t - z_t^{-1}$ is bounded on $I^+$, that we have
\[
\left| \sqrt{n}\cdot A(t) \cdot t^3 \right| \le C\cdot n^{\frac{1}{8} - \frac{3\delta}{4} + 3\varepsilon},
\]
where the exponent is negative as long as $\varepsilon < \frac{\delta_0}{4} - \frac{1}{24}$ (here we use the condition $\delta_0 > \frac{1}{6}$), and that $$
\sin\varphi_u \ge C\cdot n^{\frac{\delta}{2}-\frac{1}{4}},
$$
we arrive at the estimate
\[
\left| \; \int\limits_{I^+} e^{\theta(z-z^{-1})}\cdot z^{-x-1}\, dz \right| \; \le \; \frac{C e^{\gamma \left|\theta - \sqrt{n}\right|}}{n^{\frac{1}{8}+ \frac{\delta}{4}}}.
\]
The contribution of the arc $I^{-}$ is estimated in the same way. It remains to estimate the contributions of the circular arcs $O^+$ and $O^-$. Our aim is to show that there exists $\varepsilon_0>0$ such that
\[
\mathfrak{Re} \left( S(z) \right) < -C n^{\varepsilon_0}.
\]
for all $z \in O^+\cup O^-$. We only show it for $O^+$, as the case of $O^-$ is completely similar.

For $z \in O^+$ write
\[
z = r_z e^{i\varphi_z}.
\]
There exist positive constants $C_1, C_2, C_3$ such that
\begin{eqnarray} \label{rphi}
r_z & \ge & 1 + C_1 n^{-\frac{\delta}{4} - \frac{1}{8} + \varepsilon}, \nonumber \\
r_z & \le & 1 + C_2 n^{-\frac{\delta}{4} - \frac{1}{8} + \varepsilon}, \\
\varphi_z - \varphi_u & \ge & C_3 n^{-\frac{\delta}{4} - \frac{1}{8} + \varepsilon}. \nonumber
\end{eqnarray}
Consider the function
\[
S^\# (t) = \mathfrak{Re} S\left( te^{i\varphi_z} \right), \quad t\in[1,r_z].
\]
Note that $S^\#(1)=0$. From (\ref{rphi}) we have:
$$
\frac{dS^\#}{dt}\Big|_{t=1} \leq -C_{11} n^{\frac{\delta}{4} - \frac{3}{8} + \varepsilon}, \qquad
\left|\frac{d^2S^\#}{dt^2} \Big|_{t=1}\right| \le C_{12} n^{\frac{\delta}{2} - \frac{1}{4}}, \qquad
\max_{t \in [1,r_z]} \left| \frac{d^3 S^\#}{dt^3} \right| \le {C}_{14},
$$
whence there exists $\varepsilon_0>0$ such that
\[
S^\# (r_z) \le - C_{17}n^{\varepsilon_0}
\]
for all $z \in O^+$. We conclude that
\[
\left|\; \int\limits_{O^+} e^{\theta(z - z^{-1})} \cdot z^{-x-1} \, dz \right| \le e^{\gamma \left| \theta - \sqrt{n} \right| -Cn^{\varepsilon_0}}.
\]
The case of $O^-$ is similar, and the first claim of the Lemma is proved completely.

We proceed to the proof of the second claim. Write
\[
J_x(2\theta) - J_{x-1} (2\theta) = \frac{1}{2\pi i} \oint\limits_\mathcal K e^{\theta(z - z^{-1})} \cdot z^{-x-1}\cdot (1-z) \, dz.
\]
The contour ${\mathcal K}$ stays the same.
We first estimate the contribution of the interval $I^+$. Write
\[
1-z = 1-e^{i\varphi_u} +e^{i\varphi_u} -z,
\]
and note that, since $$\left| 1 - e^{i\varphi_u} \right| \le Cn^{\frac{2\delta-1}{4}},$$
we have
\[
\left| \; \int\limits_{I^+} e^{\theta(z-z^{-1})} \cdot z^{-x-1} \cdot \left( 1-e^{i\varphi_u} \right) \, dz \right| \; \le \;
\frac{Ce^{\gamma \left| \theta - \sqrt{n}\right|}}{n^{\frac{3}{8}-\frac{\delta}{4}}}.
\]
This is the main contribution. We proceed to the analysis of the remaining terms and estimate
\[
\int\limits_{I^+} e^{\theta (z - z^{-1})} \cdot z^{-x-1}\cdot \left( e^{i\varphi_u}-z\right) \, dz.
\]
Write
\[
e^{\theta (z - z^{-1})} \cdot z^{-x-1} = e^{\theta (z - z^{-1} -u\log z)}\cdot e^{(\theta - \sqrt{n}) u \log z}.
\]
The notation $t, \,z_t$ and $A(t)$ has the same meaning as before, and we write
\begin{multline} \label{oddity}
e^{\theta( z_t-z_t^{-1} - u\log z_t )} \cdot e^{(\theta - \sqrt{n}) u \log z_t} \cdot \left( e^{i \varphi_u} - z_t \right) =\\
= e^{-\theta \sin{\varphi_u} \cdot t^2} \left( e^{i (\theta - \sqrt{n}) u\varphi_u} \right) \cdot \frac{1+i}{\sqrt{2}} e^{i \varphi_u} \cdot t \; +\\
+ \; e^{-\theta \sin \varphi_u \cdot t^2} \left( e^{(\theta- \sqrt{n}) u \log z_t + A(t) \cdot t^3} - e^{i (\theta - \sqrt{n})u \varphi_u} \right) \times \frac{1+i}{\sqrt{2}} \, e^{i\varphi_u} \cdot t.
\end{multline}
Recall that
\[
dz_t = \frac{1+i}{\sqrt{2}} e^{i\varphi_u} \, dt
\]
and note that the first summand in the right hand side of (\ref{oddity}) is an {\it odd} function of $t$, so its integral over $I^+$ is zero. We now estimate the second summand in absolute value. Note that
\[
e^{(\theta - \sqrt{n})u\log z_t + A(t)\cdot t^3} \Big|_{t=0} = e^{i(\theta - \sqrt{n}) u \varphi_u}
\]
and that on $I^+$ we have an estimate
\[
\left| \frac{d}{dt} \, e^{(\theta - \sqrt{n})u\log z_t + A(t)\cdot t^3} \right| \le C_{20}e^{\gamma \left| \theta - \sqrt{n}\right|}.
\]
 Furthermore,
\[
\left| e^{-\theta \sin \varphi_u \cdot t^2} \right| \le e^{-\frac{\sqrt{n}}{2}\sin \varphi_u \cdot t^2}.
\]
In view of all the above, we have
\[
\left| \, e^{-\theta \sin \varphi_u \cdot t^2} \left( e^{(\theta-\sqrt{n}) u \log z_t + A(t) \cdot t^3} -e^{i (\theta-\sqrt{n}) u \varphi_u} \right) \right| \le C e^{-\frac{\sqrt{n}}{2} \sin \varphi_u \cdot t^2} \cdot \left| t \right| \cdot e^{\gamma \left| \theta - \sqrt{n} \right|}.
\]
Noting that
\[
\int\limits_{\left| t \right| \le n^{\frac{\delta}{2}-\frac{1}{4}+\varepsilon}} e^{-\frac{\sqrt{n}}{2}\sin \varphi_u \cdot t^2} \cdot t^2 \, dt \le \frac{C}{n^{\frac{3}{8}+\frac{3\delta}{4}}}\, ,
\]
we conclude that the contribution of the second summand in (\ref{oddity}) is bounded above by
\[
\frac{Ce^{\gamma \left| \theta - \sqrt{n} \right|}}{n^{\frac{3}{8} + \frac{3\delta}{4}}},
\]
and so is negligible compared to the main contribution.
Finally, we have
\[
\left| \;\int\limits_{I^+} e^{\theta (z - z^{-1})}\cdot z^{-x-1} \cdot (1-z) \, dz \, \right| \, \le \, \frac{C e^{\gamma \left| \theta - \sqrt{n} \right|}}{n^{\frac{3}{8}-\frac{\delta}{4}}}.
\]
The contribution of $I^-$ is estimated in the same way, and the contribution of the circular arcs is shown to be negligible in exactly the same way as in the proof of the first Claim. The Lemma is proven completely.

\subsection{Proof of Lemma \ref{besselmain}.}

We start with Okounkov's integral formula for the discrete Bessel kernel \cite{Okounkov}. We take any positive numbers $\alpha_1> \alpha_2 >0$ and write
\begin{equation} \label{okbes}
J(x,y;\, \theta^2) = \frac{1}{(2\pi i)^2}
\int\limits_{\left| z\right| = \alpha_1} \int\limits_{\left| w \right| = \alpha_2}
\frac{e^{\theta(z- z^{-1}-w+w^{-1})}}{(z-w) z^{x+1}w^{-y}}\, dz\,du.
\end{equation}
As before, we define $\varphi_u$ by formula $2\cos \varphi_u =u$, and we set
\[
S(z,u) = z-z^{-1} - u\log z
\]
(the principal branch of the logarithm is taken here). Setting $x=y=u\sqrt{n}$, we rewrite (\ref{okbes}) as follows
\[
J(x,x;\, \theta^2) = \frac{1}{(2\pi i)^2}
\int\limits_{\left| z\right| = \alpha_1} \int\limits_{\left| w \right| = \alpha_2}
\frac{e^{\theta(S(z,u)- S(w,u)) + (\theta - \sqrt{n})u(\log z - \log w)}}{(z-w) z}\, dz\,dw.
\]
Now, following Okounkov \cite{Okounkov}, we deform the contour of integration and obtain an integral representation for the quantity
\[
J(x,x; \, \theta^2) - \frac{\varphi_u}{\pi} \;=\;
J(x,x; \, \theta^2) - \frac{\arccos{\frac{x}{2\sqrt{n}}}}{\pi}.
\]
We take $\varepsilon >0$ sufficiently small and introduce the intervals $I_z^+$, $I_z^-$ as before:
\[
I^{\pm}_z = e^{\pm\, i \varphi_u} + \frac{1\pm i}{\sqrt{2}} e^{\pm\, i\varphi_u} \cdot t, \qquad \left| t \right| \le n^{\frac{1}{8} - \frac{\delta}{4} +\varepsilon}
\]
As before, we complete the contour by drawing a circular arc $O_z^+$ joining the outer endpoints of $I^\pm_z$ and a circular arc $O_z^-$ joining the inner endpoints of $I_z^\pm$. The resulting contour is denoted $\mathcal K_z$.

We now introduce the intervals $I_w^+, \, I_w^-$ by the formulas
\begin{alignat*}{3}
I_w^+&\;=\; e^{i\varphi_u} &\;=\; \frac{1-i}{\sqrt{2}} \,e^{i\varphi_u} \,t  \qquad\quad &\left| t\right| \le n^{-\frac{1}{8} - \frac{\delta}{4} + \varepsilon}; \\
I_w^-&\;=\; e^{-i\varphi_u} &\;=\; \frac{1+i}{\sqrt{2}} \,e^{i\varphi_u} \,t \qquad\quad &\left| t\right| \le n^{-\frac{1}{8} - \frac{\delta}{4} + \varepsilon}.
\end{alignat*}
In a similar way, we join the outer and the inner endpoints of $I^\pm_w$, respectively, by circle arcs $O_w^+$ and $O_w^-$.
\begin{figure}[h!]
\centering
\includegraphics{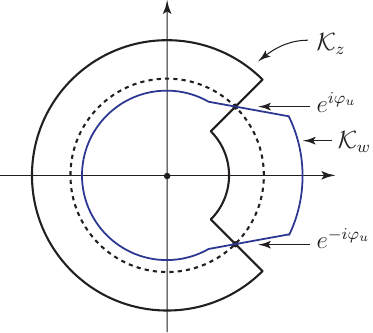}
\caption{}
\label{fig:Fig02}
\end{figure}

Okounkov \cite{Okounkov} showed that
\begin{equation} \label{okbes2}
J(x,x; \, \theta^2) -\frac{\arccos \frac{x}{2\sqrt{n}}}{\pi} \quad=\quad \frac{1}{(2\pi i)^2}
\int\limits_{\mathcal K_z}\int\limits_{\mathcal K_w}
\frac{e^{\theta\left( S(z,u) - S(w,u) \right)+ (\theta - \sqrt{n})u (\log z - \log w)}}{z(z-w)} \, dz\,dw.
\end{equation}
We estimate the right-hand side of (\ref{okbes2}), and we begin by estimating
\begin{equation} \label{rect}
\int\limits_{I^+_z}\int\limits_{I^+_w}
\frac{e^{\theta\left( S(z,u) - S(w,u) \right)+ (\theta - \sqrt{n})u (\log z - \log w)}}{z(z-w)} \, dw\,dz.
\end{equation}
As before, we write
\begin{alignat*}{2}
z_t &\;=\; e^{i\varphi_u} + \frac{1+i}{\sqrt{2}} \,e^{i\varphi_u} \,t,\\
w_s &\;=\; e^{i\varphi_u} +  \frac{1-i}{\sqrt{2}} \,e^{i\varphi_u} \,s,
\end{alignat*}
and
\begin{alignat*}{3}
S(z_t, u) &= -&&\sin \varphi_u \cdot t^2 \;+\; A(t) \cdot t^3,\\
S(w_s,u) &= &&\sin \varphi_u \cdot s^2 \;+\; \widetilde A(s) \cdot s^3.
\end{alignat*}

We set
\[
I(n) = \left[ -n^{-\frac{1}{8} -\frac{\delta}{4} + \varepsilon}, n^{-\frac{1}{8} -\frac{\delta}{4} + \varepsilon} \right]
\]
and estimate the integral
\[
\text{Int}(n) \;=\; \int\limits_{I(n)} \int\limits_{I(n)}
\frac{e^{-\theta \sin \varphi_u (t^2+s^2) - \theta ( A(t) \cdot t^3 + \widetilde A(s) \cdot s^3) + (\theta - \sqrt{n}) u (\log z_t - \log w_s)}}
{(t+is)\left( 1+\frac{1+i}{\sqrt{2}}\,t\right)} \, ds\,dt.
\]
Since $I(n)$ is symmetric around the origin, we have
\[
\int\limits_{I(n)} \int\limits_{I(n)}
\frac{e^{-\theta \sin \varphi_u (t^2 + s^2)}}{t+is} \,ds\,dt = 0,
\]
whence
\[
\text{Int}(n) \;=\; \int\limits_{I(n)} \int\limits_{I(n)}
\frac{e^{-\theta \sin \varphi_u (t^2+s^2)}}{(t+is)}
\left[ \frac{e^{( A(t) \cdot t^3 + \widetilde A(s) \cdot s^3) + (\theta - \sqrt{n})  (\log z_t - \log w_s)}}
{1+\frac{1+i}{\sqrt{2}}\,t} - 1 \right]\, ds\,dt.
\]
We estimate the integrand in absolute value. It is clear that
\[
\left| \frac{\partial}{\partial t} \left(
\frac{\exp\left(-\theta( A(t) \cdot t^3 + \widetilde A(s) \cdot s^3) + (\theta - \sqrt{n})  (\log z_t - \log w_s)\right)}
{1+\frac{1+i}{\sqrt{2}}\,t}
\right) \right|
\;\le\; C e^{-\,\gamma \left| \theta - \sqrt{n} \right|},
\]
\[
\left| \frac{\partial}{\partial s} \left(
\frac{\exp\left(-\theta( A(t) \cdot t^3 + \widetilde A(s) \cdot s^3) + (\theta - \sqrt{n})  (\log z_t - \log w_s)\right)}
{1+\frac{1+i}{\sqrt{2}}\,t}
\right) \right|
\;\le\; C e^{-\,\gamma \left| \theta - \sqrt{n} \right|},
\]
\[
\left| \frac{\exp\left(-\theta \sin{\varphi_u}\cdot (s^2 + t^2)\right)}{t+is} \right| \;\le\;
\frac{\exp\left(\frac{-\sqrt{n}}{2}\sin{\varphi_u} \cdot (s^2+t^2)\right)}{\sqrt{t^2+s^2}},
\]
whence
\[
\big| \text{Int}(n) \big| \;\le\; Ce^{\gamma \left| \theta - \sqrt{n} \right|}\cdot
\int\limits_{I(n)} \int\limits_{I(n)}
\frac{e^{\frac{-\sqrt{n}}{2}\sin \varphi_u \cdot (t^2+s^2)}}{\sqrt{t^2+s^2}} \left( \left| t \right| + \left| s\right| \right) \, dt\,ds \;\le\;
\frac{\widetilde C e^{\gamma \left| \theta - \sqrt{n}\right|}}{2 \sqrt{n} -x}.
\]
The contributions of the integrals along the remaining rectangular arcs are estimated in the same way, whereas the contribution of those parts of contours where either $z \in O_z^\pm$ or $w \in O_w^\pm$ is immediately seen to be majorated by
\[
C \exp({\gamma_1 \left| \theta - \sqrt{n} \right| - \gamma_2 n^{\delta_2}})
\]
with $\delta_2>0$ depending only on $\delta$. Lemma \ref{besselmain} is proved completely.

Theorem \ref{main} is proved completely.

\end{document}